\documentclass{amsart}
\usepackage{amssymb}
\usepackage{mathdots}
\usepackage[...]{youngtab}
\usepackage{tikz}
\usetikzlibrary{shapes,snakes}
\usepackage{framed}
\usepackage{tikz-cd}
\usepackage{amsmath}
\usepackage{txfonts}
\usepackage{booktabs}
\usepackage{listings}

\input xy
\xyoption{all}

\usepackage{amsfonts}
\usepackage{mathrsfs}
\usepackage{adjustbox}

\usepackage{mathrsfs}

\usepackage{latexsym}
\usepackage{graphicx}
\usepackage{amscd,amssymb,amsmath,amsbsy,amsthm}
\usepackage[all]{xy}

\usepackage{xcolor}
\definecolor{darkred}{RGB}{139,0,0}
\usepackage{hyperref}
\hypersetup{colorlinks,linkcolor={darkred},citecolor={darkred},urlcolor={darkred}} 
\usepackage{verbatim}

\usepackage{tikz-cd}
\usepackage{amsmath}
\usepackage{amsfonts}
\usepackage{amssymb}
\usepackage{pdfpages}
\usepackage{tikz}
\usepackage{asymptote}
\usetikzlibrary{decorations.markings,arrows}
\usepackage[all]{xy}
\usepackage{yfonts}
\usepackage[...]{youngtab}

\usepackage{mathrsfs}
\usepackage{latexsym}
\usepackage{graphicx}
\usepackage{amscd,amssymb,amsmath,amsbsy,amsthm}
\usepackage[all]{xy}

\newsavebox{\boxedtikzcdbox}

\newtheorem{theorem}{Theorem}[section]

\newtheorem{prop}[theorem]{Proposition}

\theoremstyle{definition}
\newtheorem{defn}[theorem]{Definition}
\newtheorem{ex}[theorem]{Example}
\newtheorem{remark}[theorem]{Remark}
\newtheorem{recall}[theorem]{Recall}

\DeclareMathOperator{\Gr}{\mathbf{Gr}}
\DeclareMathOperator{\id}{\mathbf{id}}
\DeclareMathOperator{\Mat}{\mathbf{Mat}}

\DeclareMathOperator{\G}{\mathbf{G}}
\DeclareMathOperator{\Fitt}{\mathbf{Fitt}}

\DeclareMathOperator{\FF}{\mathbb{F}}
\DeclareMathOperator{\NN}{\mathbb{N}}
\DeclareMathOperator{\ZZ}{\mathbb{Z}}
\DeclareMathOperator{\QQ}{\mathbb{Q}}
\DeclareMathOperator{\RR}{\mathbb{R}}
\DeclareMathOperator{\CC}{\mathbb{C}}
\DeclareMathOperator{\PP}{\mathbb{P}}
\DeclareMathOperator{\KK}{\mathbb{K}}

\DeclareMathOperator{\HH}{\mathbb{H}}
\DeclareMathOperator{\im}{im}

\DeclareMathOperator{\rep}{\mathbf{rep}}

\DeclareMathOperator{\cO}{\mathcal{O}}

\DeclareMathOperator{\lcm}{\mathbf{lcm}}

\DeclareMathOperator{\ttop}{\mathbf{top}}

\DeclareMathOperator{\fm}{\mathfrak{m}}
\DeclareMathOperator{\bd}{\mathbf{d}}

\DeclareMathOperator{\SL}{\mathbf{SL}}
\DeclareMathOperator{\GL}{\mathbf{GL}}

\DeclareMathOperator{\Spec}{\mathbf{Spec}}
\DeclareMathOperator{\Proj}{\mathbf{Proj}}
\DeclareMathOperator{\rank}{\mathbf{rank}}
\DeclareMathOperator{\gr}{\mathbf{\mathfrak{gr}}}

\newcommand{\cV}{\mathcal{V}}

\newcommand{\N}{\mathbb{N}}

\newcommand{\A}{\mathbb{A}}
\newcommand{\bk}{\mathbb{K}}
\newcommand{\scO}{\mathscr{O}}

\DeclareMathOperator{\Frac}{\mathbf{Frac}}
\DeclareMathOperator{\Hom}{\mathbf{Hom}}

\DeclareMathOperator{\init}{\mathbf{in}}

\DeclareMathOperator{\sgn}{\mathbf{sign}}

\DeclareMathOperator{\Sym}{\mathcal{\mathbf{Sym}}}
\DeclareMathOperator{\cF}{\mathcal{\mathbf{F}}}

\DeclareMathOperator{\fR}{\mathfrak{R}}

\DeclareMathOperator{\fp}{\mathfrak{p}}
\DeclareMathOperator{\fa}{\mathfrak{a}}

\newcommand{\la}{\langle}
\newcommand{\ra}{\rangle}

\newcount\cols
{\catcode`,=\active\catcode`|=\active
	\gdef\Young(#1){\hbox{$\vcenter
			{\mathcode`,="8000\mathcode`|="8000
				\def,{\global\advance\cols by 1 &}%
				\def|{\cr
					\multispan{\the\cols}\hrulefill\cr
					&\global\cols=2 }%
				\offinterlineskip\everycr{}\tabskip=0pt
				\dimen0=\ht\strutbox \advance\dimen0 by \dp\strutbox
				\halign
				{\vrule height \ht\strutbox depth \dp\strutbox##
					&&\hbox to \dimen0{\hss$##$\hss}\vrule\cr
					\noalign{\hrule}&\global\cols=2 #1\crcr
					\multispan{\the\cols}\hrulefill\cr%
				}
			}$}}
}

\title[Multiparameter Persistent Homology: Generic Structures and Quantum Computing]
{Multiparameter Persistent Homology: Generic Structures and Quantum Computing}

\author[A.~Schreiber]{Amelie~Schreiber}
\email{amelie.schreiber.math@gmail.com}


\date{\today}

\begin{document}
	
	\begin{abstract}
	The following article is an application of commutative algebra to the study of multiparameter persistent homology in topological data analysis. In particular, the theory of finite free resolutions of modules over polynomial rings is applied to multiparameter persistent modules. The generic structure of such resolutions and the classifying spaces involved are studied using results spanning several decades of research in commutative algebra, beginning with the study of generic structural properties of free resolutions popularized by Buchsbaum and Eisenbud. Many explicit computations are presented using the computer algebra package Macaulay2, along with the code used for computations. This paper serves as a collection of theoretical results from commutative algebra which will be necessary as a foundation in the future use of computational methods using Gröbner bases, standard monomial theories, Young tableaux, Schur functors and Schur polynomials, and the classical representation theory and invariant theory involved in linear algebraic group actions. The methods used are in general characteristic free and are designed to work over the ring of integers in order to be useful for applications and computations in data science. As an applications we explain how one could apply 2-parameter persistent homology to study time-varying interactions graphs associated to quadratic Hamiltonians such as those in the Ising model or Kitaev's torus code and other surface codes.
	\end{abstract}

\maketitle

\textcolor{darkred}{\hrule}

	\tableofcontents
\textcolor{darkred}{\hrule}
\medskip

%
\section{Historical Background and Motivations}

\subsection{Introduction}
The theory of multiparameter persistence arises in many contexts in computational topology and data science. Often a finite point cloud $P \subset \RR^d$ is obtained from some observational data, and one would like to understand the significant features of the data cloud. It is typical that the data set will have some \emph{"noise"}, and such noise may need to be filtered out. It is in general a very difficult problem to analyze a data set and determine what subset of the data is noise and what is a significant feature of the data. In this case, one can use \emph{topological data analysis} and \emph{persistent homology}. This method has proven very robust and effective in applications, but is a relatively new field.

The main object of study in persistent homology are the \emph{"persistent modules"}. It was established by Carlsson and Zomorodian in \cite{CZ1, CZ2} that \emph{multiparameter persistent modules} correspond to modules over a polynomial ring $\KK[x_1, ..., x_n]$, $n$ being the number of parameters. In \cite{CZ1} it is stated that certain coefficient fields (or rings) $\KK$ may be more desirable in certain cases than others. They state having a theory and methods for understanding persistence modules over polynomial rings with arbitrary coefficients is important for applications since continuous invariants may be more difficult computationally and in some applications it may be desirable to have coefficients in finite fields, fields of nonzero characteristic, or the integers $\ZZ$. 

This paper is primarily a survey, much like that of \cite{HOST}, which we will make some connections to. Results and methods from commutative algebra and algebraic geometry are framed in the setting of multiparameter persistent homology and topological data analysis. The applications come from research initiated in the early 1970s by Buchsbaum and Eisenbud on the structure of free resolutions of modules over commutative rings, especially the theory of so-called \emph{generic free resolutions}. 

In this paper, we will establish a structure theory for the free resolutions of multiparameter persistent modules, which are finitely generated modules over the polynomial ring $\KK[x_1, ..., x_n]$, where $\KK$ may be any field, or the ring of integers $\ZZ$. The study of free resolutions of modules, especially that of the generic structure of resolutions over Noetherian rings has been extensively studied for decades (see for example \cite{Ho} and \cite{BE1, BE2, BE3}). The quest for a complete and practical understanding of the generic structure of free resolutions and the related representation theory of algebraic groups acting on rings is an endeavor which is still an active area of research (see for example \cite{B}, \cite{W} 
). In the following paper, many of these techniques and results are used to study multiparameter persistent homology. Ample references to past research on the subject are given.

\subsection{Remarks on Background Material and Assumptions in the Paper}
Throughout, the ring $R = \KK[x_1, ..., x_n]$ will always be the polynomial ring with coefficients in an arbitrary field or in the integers. The results which follow are almost entirely independent of the characteristic of the base field (or ring), and one may assume we are working over $\ZZ$ unless specifically stated otherwise. This is done for computational purposes, as computing over the field $\CC$ or even $\QQ$ or its algebraic closure $\overline{\QQ}$ can be computationally expensive, and can give rise to large errors when computing with large data sets. For example, it is shown in \S $1.4.1$ in \cite{C} that even computations with Hilbert matrices of relatively small size lead to large errors rather quickly. In certain specific cases, we may need to assume the algebraic closure of the base field. We will state this assumption when it is necessary, but for the majority of the results not even this assumption is required. 

There will be applications of the representation theory of algebraic groups, especially of $\GL(n)$ and $\SL(n)$, the \emph{general} and \emph{special linear groups}. Because much of what follows is done in the generality of coefficients over $\ZZ$, we will need a representation theory of these groups which works in this generality. This will require notions such as \emph{schemes} and \emph{functors of points}, which we will review the basics of. For the details and background we refer the reader to \cite{EH} and \cite{M}. 

We will need this, for example, when constructing \emph{universal objects} such as the \emph{Grassmannian functor}. The Grassmannian over a field is a classical construction, but if one wishes to work over arbitrary base rings, for example $\ZZ$, or a polynomial ring $\KK[x_1, ..., x_n]$, one needs something more general than the classical construction. In the case of the Grassmannian, this construction is not hard. It turns out if one constructs the Grassmannian $\Gr_{\ZZ}(r,k)$ over $\ZZ$, then for any other base scheme $S = \Spec(A)$, for $A$ any commutative ring, we have
\[ \Gr_S(r,k) = \Gr_{\ZZ}(r,k) \times S.\]
More generally, if we have any homomorphism of affine schemes $\phi: T \to S$ we have
\[ \Gr_T(r,k) = \Gr_S(r,k) \times_S T \]
is given by a fibre product. We will also talk about the \emph{universal complexes}, which is a universal object for complexes of free modules over a commutative ring in the same way $\Gr_{\ZZ}(r,k)$ is for Grassmannians. From this comes the study of \emph{generic free resolutions}, which are in some sense "\emph{the most general}" free resolutions in a way we will make precise later on. They give a kind of classifying space for free resolutions of a certain specified type. 

At this point, we will also need the notion of Schur functors and the combinatorics of Young tableaux in order to describe the irreducible representations of these linearly reductive groups. We will review this material as well, but refer the reader to \cite{E}, \cite{F}, \cite{FH}, \cite{DEP}, \cite{DS}, \cite{T1}. We will look at some specific results on free resolutions where it is assumed that the base field has characteristic zero. So, for example, one may assume we are working over the $\fp$-adic numbers $\QQ_{\fm}$. As seen in \cite{C}, this can still be useful from a computer science perspective. We will also state which specific cases require this particular assumption.

\subsection{Organization of the Paper}
In Sections \ref{Young tableaux} and \ref{MP persistence} we review the theory of Schur functors and Young tableaux, and recall basic information on multiparameter persistence homology and related commutative algebra and algebraic geometry. Our first task once this is done will be to develop a generic structure theory in order to extend the theory of parametrizing spaces of persistence modules started in \cite{CZ1}, where the so-called \emph{relation families} are studied. As it turns out, these are simply presentation spaces of multiparameter persistence modules. In this case the theory of \emph{determinantal ideals} and \emph{determinantal varieties} are of great utility. We will present the theory of the so-called \emph{Fitting invariants} of persistence modules and use them to study presentations of persistence modules. This is the content of Section \ref{presentation spaces}. Next in Section \ref{Generic Structure} we will define \emph{generic free resolutions} of persistence modules and we will review Bruns' \emph{"generic exactification"} of a complex of modules. Once the generic structure is described, we introduce the \emph{Buchsbaum-Eisenbud multipliers} and use these to state some deeper results about generic free resolutions of persistence modules in Section \ref{BE multipliers}. We then come to the \emph{varieties of complexes} in Section \ref{varieties of complexes}, where we study a scheme which gives geometric properties of families of free resolutions of persistent modules. In Section \ref{Bases and Standard Monomials} we then give a Gr\"{o}bner basis of the coordinate rings of varieties of complexes, and we develop a \emph{standard monomial theory} using the combinatorics of \emph{Young tableaux}. This section in particular provides methods which are particularly useful tools standard in computational algebraic geometry. Throughout we provide many explicit examples, many of which were computed using Macaulay2. We provide the code for the computations carried out in the examples in the appendix at the end.

%
\section{Young Tableaux and Schur Functors}\label{Young tableaux}
A \textbf{partition} $\lambda \vdash m$, of some nonnegative integer $m$ is a sequence of nonincreasing numbers $\lambda = (\lambda_1, ..., \lambda_s)$ such that $\sum_i \lambda_i = m$. We define a \textbf{Young diagram} corresponding to a partition $\lambda = (\lambda_1, ..., \lambda_s) \vdash m$, as the diagram with $\lambda_i$ boxes in the $i^{th}$ row. We identify partitions with their Young diagrams and speak of the two interchangeably. For a free module $E$ over a commutative ring $\KK$ with ordered basis $\{e_1, ..., e_n\}$, we associate a filling of the diagram $\lambda$ by integers $\{1, ..., n\}$ to an element in the module
\[ \bigwedge^{\lambda_1}E \otimes \cdots \otimes \bigwedge^{\lambda_s}E \]
as follows: If in row $i$ and box $j$ of $\lambda$ we have the integers $t(i, j)$ for $j=1, ..., \lambda_i$, we associate the element
\[ e_{t(1,1)} \wedge \cdots \wedge e_{t(1,\lambda_1)} \otimes  \cdots \otimes e_{t(s,1)} \wedge \cdots \wedge e_{t(s, \lambda_s)}.\]
We define such a filling to be a \textbf{tableaux}. A tableaux is \textbf{standard} if its rows are strictly increasing, and its columns are nondecreasing. 

Fix a free module $E$ of rank $n$ over a commutative ring $K$. Let $\lambda = ( \lambda_1, ..., \lambda_s) \vdash m$ be a partition. We define the module
\[ L_{\lambda}E = \bigwedge^{\lambda_1} \otimes \cdots \otimes \bigwedge^{\lambda_s}E/R(\lambda, E) \]
where the submodule $R(\lambda, E)$ is the sum of all submodules of the form
\[ \bigwedge^{\lambda_1}E \otimes \cdots \otimes \bigwedge^{\lambda_{a-1}}E \otimes R_{a, a+1}E \otimes \bigwedge^{\lambda_{a+2}}E \otimes \cdots \otimes \bigwedge^{\lambda_s}E \]
for $1 \leq a \leq s-1$. Here $R_{a,a+1}E$ is the submodule spanned by the images of the maps $\theta(\lambda, a, u, v, E)$
\[ \xymatrix{ \bigwedge^uE \otimes \bigwedge^{\lambda_a-u+\lambda_{a+1}-v}E \otimes \bigwedge^vE \ar[d]^{1 \otimes \Delta \otimes 1} \\
	\bigwedge^uE \otimes \bigwedge^{\lambda_a-u}E \otimes \bigwedge^{\lambda_{a+1}-v}E \bigwedge^vE \ar[d]^{m_{12} \otimes m_{34}} \\
	\bigwedge^{\lambda_a}E \otimes \bigwedge^{\lambda_{a+1}}E}.\]
Here, $\Delta$ is the diagonal embedding (or comultiplication), and $m: \bigwedge^r E \otimes \bigwedge^s E \to \bigwedge^{r+s}E$ is a component of the multiplication map on the exterior algebra $\bigwedge^{\bullet}E$. Combinatorially, we may use the theory of Young tableaux to describe the modules $L_{\lambda}E$. It is well known the standard 
tableaux form a basis of $L_{\lambda}E$. Further, the relations $R(\lambda, E)$ may be visualized as follows. Let $\lambda \vdash 
m$ be a partition with at most $n$ parts. To the map $\theta(\lambda, a, u, v, E)$ we associate a \textbf{Young scheme}. The Young 
scheme is defined as being a diagram of shape $\lambda$ with empty boxes everywhere except the $a^{th}$ and 
$a+1^{st}$ rows. In row $a$ there are $u$ empty boxes followed by $\lambda_a-u$ filled boxes. In row $a+1$ there are 
$\lambda_{a+1}-v$ filled boxes followed by $v$ empty boxes. We restrict to $u+v<\lambda_{a+1}$ as in the definition of 
$L_{\lambda}E$. Let $U_j = e_{t(j,1)} \wedge \cdots \wedge e_{t(j, \lambda_j)}$, $V_1 = e_{x_1} \wedge \cdots \wedge e_{x_u}$, 
$V_2 = e_{y_1} \wedge \cdots \wedge e_{y_{\lambda_a-u+\lambda_v-v}}$, and $V_3 = e_{z_1} \wedge \cdots \wedge e_{z_v}$. 
Then the image of a typical element 
\[ U_1 \otimes \cdots \otimes U_{a-1} \otimes V_1 \otimes V_2 \otimes V_3 \otimes U_{a+2} \otimes \cdots \otimes U_s \]
is a sum of tableaux where we put $x_1, ..., x_u$ in the empty $u$ boxes in row $a$, and $z_1, ..., z_v$ in the empty $v$ boxes of row $a+1$, and we shuffle the elements $y_1, ..., y_{\lambda_a-u+\lambda_{a+1}-v}$ between the filled boxes in row $a$ and $a+1$. The coefficients of the tableaux in the summation are $\pm 1$ depending on the sign of the permutations coming from the exterior diagonals. 







\begin{remark}
	From now on, we will let $L_{\lambda}F$ denote the Schur functor corresponding to the diagram $\lambda$, and we will let $S_{(\lambda-k)}F = L_{\lambda}F \otimes (\bigwedge^{\dim F}F^*)^{\otimes k}$, where $\lambda-k$ is obtained by subtracting the integer $k$ from every entry of the vector $\lambda$. 
\end{remark}

%
\section{Multiparameter Persistence Modules}\label{MP persistence}

\subsection{Homological Persistence}
The theory of multiparameter persistence arises in many contexts in computational topology and data science. Often a point finite cloud $P \subset \RR^d$ is obtained from some observational data, and one would like to understand the significant features of the data cloud. It is typical that the data set will have some \emph{"noise"}, and such noise may need to be filtered out. It is in general a very difficult problem to analyse a data set and determine what subset of the data is noise and what is a significant feature of the data. In this case, one can use topological data analysis and persistence homology. This method has proven very robust and effective in applications.

\begin{defn}
A \textbf{multifiltered space} $X$, is a topological space with a family of subspaces $\{X_v \subseteq\}_{v \in \NN^n}$, with inclusion maps $X_u \hookrightarrow X_v$ whenever $u \leq v$ in the standard partial order on $\NN^n$. We require the following diagram to commute:
\[ \xymatrix{
X_u \ar[r] \ar[d] & X_{v_1} \ar[d] \\
X_{v_2} \ar[r] & X_w}
\]
whenever $u \leq v_1, v_2 \leq w$. 
\end{defn}

\begin{defn}
Let $\KK$ be a field. A \textbf{persistence module} $M$, is a family of $\KK$-modules $\{M_u\}_{u \in \NN^n}$ with maps 
\[ \phi_{u,v}: M_u \to M_v \]
for all $u \leq v$ such that $\phi_{v,w} \circ \phi_{u,v} = \phi_{u,w}$ for all $u \leq v \leq w$. Let $M$ be a persistence module and let $R = \KK[x_1, ..., x_n]$. Define
\[ \alpha(M) = \bigoplus_v M_v \]
where the $\KK$-module structure is the direct sum structure, and $x^{u-v}:M_u \to M_v$ is $\phi_{u,v}$ when $u \leq v$ in $\NN^n$. This gives an equivalence of categories between the category of finite persistence modules over $\KK$, and the category of $\NN^n$-graded finitely generated modules over $R$. 
\end{defn}

Now, the first two terms of a minimal free resolution of a persistence module $M$ are unique up to isomorphism of free chain complexes. This is important for classification of isomorphism classes of multigraded persistence modules. 


\subsection{Free Resolutions}
The structure of free resolutions of modules over commutative rings is one of the most fundamental objects of study in commutative algebra. If one wants to understand a ring well, one needs to understand the modules over that ring. If one has some structure theorems about what free (or projective/injective) resolutions of the modules "look like", then one understands the modules quite well. In some very classical results which still have some significant implications which are being studied in commutative algebra, Buchsbaum and Eisenbud gave their "structure theorems" on free resolutions over Noetherian rings (see \cite{BE1, BE2, BE3} for example). 

It was a generalization of classical results of Hilbert, and later Burch, which is unsurprisingly known as the \textbf{Hilbert-Burch Theorem}. Buchsbaum and Eisenbud's results were of course extended, notably by Eagon and Northcott \cite{EN}. A very nice summary of the results can be found in the notes of Hochster.

For the time being let us focus on the "simple" examples which relate to bifiltrations in Topological Data Analysis, i.e. $\mathbb{N}^2$-graded modules over $R=\KK[x,y]$. Many of the following statements are independent of characteristic, and so unless otherwise stated, one may assume $\KK$ to be an arbitrary field, or the ring of integers.

Let us attempt to put the results of Buchsbaum and Eisenbud into simple and straightforward language for those who are not experts at commutative algebra, and to provide some concrete examples as an introduction. The precise statements and proofs will be provided in subsequent sections. In the $\mathbb{N}^2$-graded case, persistence modules are modules over $\KK[x,y]$, and the free resolutions are at worst of length $2$. So in general, we are in the case elaborated on by Hochster, where we study free resolutions of the form
\[ 0 \to R^{b_2} \to R^{b_2} \to R^{b_0} \]
where we have excluded the module given by the cokernel of the right most map to $R^{b_0}$. Let us look at an example which comes from the paper \cite{HOST}. In the paper this resolution shows up in several examples and is presented in graded form incorrectly. In the notation of \cite{HOST}, the graded resolution should look like, 

\begin{tikzcd}[ampersand replacement=\&, column sep=.35in,row sep=1in]
S(-2,-3) \arrow{rr}{ \begin{pmatrix} 0 \\ -y \\ x \end{pmatrix}_{d_2}} \& \& \begin{matrix} S(-3,-1) \\ \oplus S(-2,-2) \\ \oplus S(-1,-3) \end{matrix} \arrow{rr}{ \begin{pmatrix} x^2 & 0 & 0\\ 0 & x & y \end{pmatrix}_{d_1}} \&  \& \begin{matrix}S(-1,-1) \\ \oplus S(-1,-2)\end{matrix} \arrow{rr}{ p_M } \&  \& \begin{matrix}S(-1,-1)/(x^2) \\ \oplus S(-1,-2)/(x,y) \end{matrix}
\end{tikzcd}

and we have thrown out the "extraneous" part of the resolution involving the free summand $S(-2,-2)$ of the module
\[ H_1(K') = S(-2,-2) \oplus S(-1,-1)/x^2 \oplus S(-1,-2)/(x,y)\]
and we have rewritten the resolutions of $S/x^2$ and $S/(x,y)$ provided in their paper in the appropriate grading, which may have proved confusing for someone not familiar with commutative algebra. We also leave out the zeros simply for formatting and lack of space. For those who are new to commutative algebra, $S(-i,-j) \cong k[x,y]$ in the above resolution is the "multi-graded shift" of the ring $R=\KK[x,y]$. So in our notation, suppressing the shifts we have, 
\[
\xymatrix{
0 \ar[r] & R^1 \ar[r]^{\begin{pmatrix}
0 \\ -y \\ x
\end{pmatrix}}_{d_2}
& R^3 \ar[rr]^{\begin{pmatrix}
x^2 & 0 & 0\\
0 & x & y
\end{pmatrix} }_{d_1}
& & R^2 \ar[r]^{p_M \quad \quad \quad} & R/(x^2) \oplus R/(x,y) \ar[r] & 0
}
\]
Now, in general we have a natural isomorphism from standard multilinear algebra,
\[ \bigwedge^k R^n \cong \bigwedge^{n-k} (R^n)^* \]
so in this example, 
\[ \bigwedge^2R^3 \cong \bigwedge^1R^3 \]
\[
\bigwedge^2 d_1 = \begin{pmatrix}
x^3 & x^2y & 0
\end{pmatrix} \quad \quad \bigwedge^{1} d_2 = d_2 = \begin{pmatrix}
0 \\ -y \\ x
\end{pmatrix}
\]
If we let $[i,j]_1$ denote the $2 \times 2$ minor of $d_1$ given by the columns $i$, and $j$, and we let $k$ be the complement of $i, j \in \{1,2,3\}$, so that $[k]_2$ is the $1 \times 1$ minor of $d_2$, we have the following equations
\[ [1,2]_1 = a[3]_2, \quad [1,3]_1 = a[2]_2, \quad [2,3]_1 = a[1]_2
\]
Then the "Buchsbaum-Eisenbud" multiplier is $x^2 \in R$. The "most generic resolution" with format $(1,3,2)$ is generated over $\mathbb{Z}$ by the entries of two "generic" matrices of the same size as $d_1$ and $d_2$, and the Buchsbaum-Eisenbud multiplier. If we let $\varphi = (x_{i,j})$ and $\psi = (y_j)$, for $i, j \in \{1,2,3\}$, and let $k$ be the complement of $i, j$ in $\{1,2,3\}$, then it is a quotient of:
\[ \mathbb{Z}[x_{1,1}, x_{1,2}, x_{1,3}, x_{2,1}, x_{2,2}, x_{2,3}, y_1, y_2, y_3][x^2] = \mathbb{Z}[x_{i,j}, y_j][x^2] \]
by the ideal,
\[ (\varphi \cdot \psi) + ([i,j]_{\varphi} + \mathit{sgn}(i,j,k) \cdot x^2 \cdot [k]_{\psi}) \]
This is the general setup for the $\mathbb{N}^2$ graded case. Hochster gives a wonderful explanation of this in \cite{}, but we will include a paraphrasing of it for completeness and so that we can try to keep the paper self contained. In pedestrian terms, for an arbitrary free resolution 
\[ \xymatrix{
0 \ar[r] & R^{b_2} \ar[r]^{\psi} & R^{b_1} \ar[r]^{\varphi} & R^{b_0} 
}\]
of an $\mathbb{N}^2$-graded persistence module, we can compute the ring $\mathcal{R}:= \mathbb{Z}[x_{i,j}, y_{j,k}][a_I]$, where $x_{i,j}$ and $y_{j,k}$ are "generic" functions defining the maps $\psi$ and $\varphi$. 
We then quotient out by the ideal given by the multiplication $\psi \cdot \varphi$, then we quotient out relations between the minors of the two matrices, which are given by list of $b_0 \choose r_1$ \textbf{Buchsbaum-Eisenbud multipliers}, "$a_i$". In our example, the multipliers come from the three equations
\[ [1,2]_{\varphi} = a_1[3]_{\psi}, \quad [1,3]_{\varphi} = a_2[2]_{\psi}, \quad [2,3]_{\varphi} = a_3[1]_{\psi}
\]
Notice the indices on the right side of each equality are the complement of those on the left. This is always the case, and is essentially a consequence of the isomorphism 
\[
\Gr(r,V) \cong \Gr(n-r,V^*)
\]
on Grassmannians. For the details, we require slightly more high-brow language, so avert your eyes if you are squeamish around multilinear algebra. 

For any free resolution of format $(b_2, b_1, b_0)$, say
\[ \xymatrix{
\mathbf{F}_{\bullet}:= \ 0 \ar[r] & F_2 \ar[r]^{d_2} & F_1 \ar[r]^{d_1} & F_0 
}\]
the Buchsbaum-Eisenbud multipliers come from the requirement that the following diagrams commute,
\[
\xymatrix{
\bigwedge^{r_{i+1}}F_i^* \cong \bigwedge^{r_i}F_i \ar[dr]_{a_{i-1}^*} \ar[rr]^{\bigwedge^{r_i}d_i} && \bigwedge^{r_i}F_{i-1} \\
& R \ar[ur]_{a_i}
}
\]
we identify this with the quotient of 
\[ \mathbf{Sym}\left(F_2 \otimes F_1^* \oplus F_1 \otimes F_0^*\right) \]
with the ideal generated by the representations
\[ F_2 \otimes F_0^* \hookrightarrow \left(F_2 \otimes F_1^*\right) \otimes \left(F_1 \otimes F_0^*\right) \]
corresponding to the sacred equation of the complex $d_2 \circ d_1 = 0$, 
\[ \bigwedge^{r_i+1}F_i \otimes \bigwedge^{r_i+1}F_{i-1}^* \]
corresponding to the rank condition $\bigwedge^{r_i+1}d_i = 0$ (in more generality, these results hold for when the ranks do not add up to be the rank of the module $F_1$). 

Now, denote by $\mathcal{R}_{comp}$ to be the ring above after taking quotients by these ideals. Now, what is the importance of this ring? It can be thought of as the "coordinate ring" of a "general free resolution" 
\[ \mathbb{F}_{\bullet}:= \mathbb{F}_2 \to \mathbb{F}_1 \to \mathbb{F}_0 \]
which classifies all resolutions of the format $(b_2, b_1, b_0)$. In particular, it will give an explicit description of the free resolutions of all persistence modules with resolutions of this format. In the paper by Carlsson and Zomodorian, they propose the study of a space they denote by $\mathcal{RF}(\xi_1, \xi_0)$. This is just the restriction of the above to the "general presentations" $\mathbb{F}_1 \to \mathbb{F}_0$, with fixed $b_1 = \xi_1$ and $b_0 = \xi_0$. There is some study of this in terms of quivers, and the results from commutative algebra will all be applicable as well. 

In particular, we can find all persistence modules with minimal resolutions of the same format, and therefore with the same minimal presentations via a unique map. This gives a very concrete and complete description of the spaces $\mathcal{RF}(\xi_1, \xi_0)$. One can think of this as taking the closure of the generic object, and all other objects being a degeneration of this general free resolution.

\subsection{The Grassmannian Functor}
Let $R = \KK[x_1, ..., x_n]$, and let $M$ be an $R$-module with an ordered set $m_1, ..., m_k \in M$. Define an equivalence relation
\[ (M; m_1, ..., m_k) \sim (M'; m_1', ..., m_k') \]
if and only if there is an isomorphism $fM \to M'$ such that $f(m_i) = m_i'$. Define the \textbf{Grassmanian functor}, denoted $\Gr(r, k)$, to be the functor from the category of rings to the category of sets which assigns to a ring $R$
\[ \left\lbrace \substack{ (M;m_1, ..., m_k), \ \rank(M)=r, \\ M=R\la m_1, ..., m_k\ra } \right\rbrace \bigg/ \sim\]
and to each ring homomorphism $f:R \to S$ it assigns a set map
\[ [M;, m_1, ..., m_k] \mapsto [M \otimes_R S; m_1 \otimes 1, ..., m_k \otimes 1],\]
where $[M; m_1, ..., m_k]$ denotes the equivalence classes of locally free $R$-modules. 
The importance of viewing the classical Grassmannian manifold (over an algebraically closed field) as a scheme, or 
the associated functor of points is due to the fact that constructing the Grassmannian $\Gr_{\ZZ}(r,n)$ over $
\Spec(\ZZ)$, and defining 
$$ \Gr_{S}(r,n) = \Gr_{\ZZ}(r,n) \times S $$
for any affine scheme $S = \Spec(A)$, yields the desired Grassmannian for arbitrary base rings. Said another way, for any $0<r<k$, 
let 
\[ \Gr(r,k): \mathbf{Rings} \to \mathbf{Sets} \]
be the functor given by
\[ \Gr(r,k)(S) = \{\text{rank} \ r \ \text{direct summands of} \ S^n\}. \]
Letting $I = (i_1, i_2, ..., i_r) \subseteq \{1,2,...,k\}$ be a subset with $1 \leq i_1 < i_2 < \cdots < i_r \leq k$. Let $x_I = x_{(i_1, i_2, ..., i_k)}$ be indeterminates over $\ZZ$ so that $S = \ZZ[x_I]_{I \subseteq \{1,2,...,k\}}$ be the ring over $\ZZ$ in $k\choose r$ variables. The variables $x_I$ can be thought of as corresponding to maximal minors of an $r \times k$ generic matrix $M \in \Mat_{r \times k}(\ZZ)$. Now, let $M = (\id_r, B)$ have the identity matrix as its first $r \times r$-submatrix, and $B$ is an $r \times (k-r)$ matrix. Then the maximal minors of $M$ can be identified with the minors of $B$ of all sizes. We will see some examples of this construction in the next section. The \textbf{Pl\"{u}cker relations} are the homogeneous polynomials in the variables $x_I$ given by expanding minors of complementary submatrices. In particular, if we let $\ZZ[y_{1,1}, y_{1,2}, ..., y_{r,k}] = \ZZ[Y]$, where $Y = (y_{i,j})$ is an $r \times k$ matrix of variables, and then map
\[ x_I \mapsto (i_1, i_2, ..., i_r) \]
where $(i_1, i_2, ..., i_r)$ is the minor of $Y$ given by the columns $i_1, ..., i_r$, we obtain a map
\[ \phi: \ZZ[x_I] \to \ZZ[Y] \]
with $\ker(\phi) = J$ the Pl\"{u}cker relations. Then we define the projective scheme
\[ \Gr_{\ZZ}(r,k) = \Proj\left(\ZZ[x_I]/J \right) \hookrightarrow \Proj(\ZZ[x_I]) = \PP^{{k\choose r} -1}_{\ZZ} = \PP\left(\bigwedge^r \ZZ \right). \]
In the next section we will see how the combinatorics of Young tableaux are used for Grassmannians. 

\begin{remark}
The above construction is functorial. In particular, for any affine schemes $S, T$ with a homomorphism $T \to S$ we have 
\[ \Gr_T(r,k) = \Gr_S(r,k) \times_S T \]
is given by a fibre product of schemes. 
\end{remark}


\section{Presentation Spaces, Rank Invariants, and Fitting Ideals}\label{presentation spaces}

Let $\NN^n$ be the partially ordered lattice with partial order given by
\[ u = (u_1, ..., u_n) \leq v = (v_1, ..., v_n) \ \iff \ u_i \leq v_i \ \text{for all} \ i=1, ..., n.\]
Next, define the reverse degree lexicographic order on $\NN^n$ by
\[ u < v \ \iff \ \deg(u)<\deg(v) \ \text{or} \ \deg(u) = \deg(v) \ \text{and} \ u_i>v_i, \]
where $i$ is the first index such that $u_i \neq v_i$ in reverse order counting backwards down to $i$, i.e. $n, n-1, n-2, ..., i$. For example, if $n =2$ and we look at the lattice $\NN^2$, we picture this as

\medskip

\begin{center} Figure \ref{Sym maps} \end{center}
\[
\begin{tikzcd}\label{Sym maps}
\fbox{$x^6$} \arrow[r] & x^6y \arrow[r] & x^6y^2 \arrow[r] & x^6y^3 \arrow[r] & x^6y^4 \arrow[r] & x^6y^5 \arrow[r] & x^6y^6 \\
x^5 \arrow[r] \arrow[u] & \fbox{$x^5y$} \arrow[r] \arrow[u] & x^5y^2 \arrow[r] \arrow[u] & x^5y^3 \arrow[r] \arrow[u] & x^5y^4 \arrow[r] \arrow[u] & x^5y^5 \arrow[r] \arrow[u] & x^5y^6 \arrow[u] \\
\fbox{$x^4$} \arrow[u] \arrow[r] \arrow[ru, "xy=yx" description, no head, dotted] & x^4y \arrow[u] \arrow[r] & \fbox{$x^4y^2$} \arrow[u] \arrow[r] & x^4y^3 \arrow[u] \arrow[r] & x^4y^4 \arrow[u] \arrow[r] & x^4y^5 \arrow[u] \arrow[r] & x^4y^5 \arrow[u] \\
x^3 \arrow[u] \arrow[r] & \fbox{$x^3y$} \arrow[u] \arrow[r] \arrow[ru, "xy=yx" description, no head, dotted] & x^3y^2 \arrow[r] \arrow[u] & \fbox{$x^3y^3$} \arrow[r] \arrow[u] & x^3y^4 \arrow[r] \arrow[u] & x^3y^5 \arrow[u] \arrow[r] & x^3y^6 \arrow[u] \\
x^2 \arrow[u] \arrow[r] & x^2y \arrow[u] \arrow[r] & \fbox{$x^2y^2$} \arrow[r] \arrow[u] \arrow[ru, "xy=yx" description, no head, dotted] & x^2y^3 \arrow[r] \arrow[u] & \fbox{$x^2y^4$} \arrow[r] \arrow[u] & x^2y^5 \arrow[u] \arrow[r] & x^2y^6 \arrow[u] \\
x \arrow[u] \arrow[r] & xy \arrow[u] \arrow[r] & xy^2 \arrow[r] \arrow[u] & \fbox{$xy^3$} \arrow[r] \arrow[u] \arrow[ru, "xy=yx" description, no head, dotted] & xy^4 \arrow[r] \arrow[u] & \fbox{$xy^5$} \arrow[u] \arrow[r] & xy^6 \arrow[u] \\
1 \arrow[u] \arrow[r] & y \arrow[r] \arrow[u] & y^2 \arrow[r] \arrow[u] & y^3 \arrow[r] \arrow[u] & \fbox{$y^4$} \arrow[r] \arrow[u] \arrow[ru, "xy=yx" description, no head, dotted] & y^5 \arrow[r] \arrow[u] & \fbox{$y^6$} \arrow[u]
\end{tikzcd}
\]

\medskip

Here, we identify $x^ry^s$ with the coordinate $(r,s) \in \NN^2$. In this way, we view 
\[ \KK[x,y] = \Sym(\KK^2) = \bigoplus_{n \geq 0} \Sym^n(\KK^2) \]
as a bi-graded vector space, with bi-degree $(r,s)$ for the monomial $x^ry^s$. More generally, 
\[ \KK[x_1, ..., x_n] = \Sym(\KK^n) = \bigoplus_{k \geq 0} \Sym^k(\KK^n) \]
is an $n$-graded vector space, with multi-degree $(a_1, ..., a_n)$ for the monomial $x_1^{a_1} \cdots x_n^{a_n}$.  We have vector space maps $x_i: \Sym^{(a_1, ..., a_n)}(\KK^n) \to \Sym^{(a_1, ..., a_i+1, ..., a_n)}(\KK^n)$. In somewhat different notation, this can be written more compactly as
\[
x_i : \KK(x^a) \to \KK(x^{a}x_i)
\]
where $\KK(x^a) = \KK(x_1^{a_1} \cdots x_n^{a_n}) = \Sym^{(a_1, ..., a_n)}(\KK^n) \cong \KK$ is the one dimensional subspace of $\Sym^r(\KK^n)$ spanned by the monomial $x_1^{a_1} \cdots x_n^{a_n}$, where $\sum_{i=1}^n a_i = r$. In $\NN^2$, the spaces $\Sym^4(\KK^2)$ and $\Sym^6(\KK^2)$ are boxed in Figure \ref{Sym maps}. The dotted lines indicate the map $xy: \Sym^4(\KK^2) \to \Sym^6(\KK^2)$. Now, let $u = (u_1, ..., u_n), v=(v_1, ..., v_n) \in \NN^n$ be two vectors. Let us define $\la u \ra = \la (u_1, ..., u_n) \ra$ to be the ideal generated by the monomial $x_1^{u_1} \cdots x_n^{u_n}$. For example, in the following lattice for $\NN^2$ we have the ideal $\la (3,1) \ra = R \la x^3y \ra$ boxed. 
\medskip
\[
\begin{tikzcd}
 & \vdots &  &  & \vdots &  &  \\
{(0,4)} \arrow[r] & {(1,4)} \arrow[r] & {(2,4)} \arrow[r] & {\fbox{$(3,4)$}} \arrow[r] & {\fbox{$(4,4)$}} \arrow[r] & {\fbox{$(5,4)$}} &  \\
{(0,3)} \arrow[r] \arrow[u] & {(1,3)} \arrow[r] \arrow[u] & {(2,3)} \arrow[r] \arrow[u] & {\fbox{$(3,3)$}} \arrow[r] \arrow[u] & {\fbox{$(4,3)$}} \arrow[r] \arrow[u] & {\fbox{$(5,3)$}} \arrow[u] & \cdots \\
{(0,2)} \arrow[u] \arrow[r] & {(1,2)} \arrow[u] \arrow[r] & {(2,2)} \arrow[u] \arrow[r] & {\fbox{$(3,2)$}} \arrow[u] \arrow[r] & {\fbox{$(4,2)$}} \arrow[u] \arrow[r] & {\fbox{$(5,2)$}} \arrow[u] &  \\
{(0,1)} \arrow[u] \arrow[r] & {(1,1)} \arrow[u] \arrow[r] & {(2,1)} \arrow[u] \arrow[r] & {\fbox{$(3,1)$}} \arrow[u] \arrow[r] & {\fbox{$(4,1)$}} \arrow[u] \arrow[r] & {\fbox{$(5,1)$}} \arrow[u] & \cdots \\
{(0,0)} \arrow[u] \arrow[r] & {(1,0)} \arrow[u] \arrow[r] & {(2,0)} \arrow[u] \arrow[r] & {(3,0)} \arrow[u] \arrow[r] & {(4,0)} \arrow[u] \arrow[r] & {(5,0)} \arrow[u] & 
\end{tikzcd}
\]
\medskip
Now, as a vector space, this is isomorphic to the ring $R = \KK[x_1, ..., x_n]$. It will be denoted $R(-3,-1)$. In general, we denote the free multigraded $R$-module generated by $(u_1, ..., u_n) \in \NN^n$, by $R(-u) = R(-u_1, ..., -u_n)$. There is a multigraded map of free multigraded modules $R(-u_1, ..., -u_n) \to R(-v_1, ..., -v_n)$ given by $x^{v-u}$, which corresponds to any (directed) path from $v$ to $u$ in the lattice $\NN^n$. If no such path exists, then there is no multigraded map $R(-u) \to R(-v)$. In general, we may express any multigraded map of free modules
\[ f: \bigoplus_{\substack{i=1 \\ u(i) = (u(i)_1, ..., u(i)_n)}}^{b_1} R(u(i))^{b_1(i)} \to \bigoplus_{\substack{j=1 \\ v(j) = (v(j)_1, ..., v(j)_n)}}^{b_0} R(a(i))^{b_0(j)} \]
where $b_1(i)$ is the number of summands in multidegree $u(i)$ and $b_0(j)$ is the number of summands in multidegree $v(j)$, and $b_1 = \sum_i b_1(i)$ and $b_0 = \sum_j b(j)$ are each finite. 

The way one might think of this, is as assigning to each $(u(i)_1, ..., u(i)_n) \in \NN^n$, the multiplicity $b_1(i)$ for all $u'(i) \geq u(i)$ in the partial order on $\NN^n$, and similarly for the $v(j)$. So, for example:
\newpage
\[ xy: R(-u_1-1, -u_2-1) \to R(-u_1, -u_2) \]
\[x^2: R(-u_1-2, -u_2) \to R(-u_1, -u_2) \]
\[y^2: R(-u_1, -u_2-2) \to R(-u_1, -u_2)\]

can be represented by the matrix
\[ \begin{pmatrix}
xy & 0 & 0 \\
0 & x^2 & 0 \\
0 & 0 & y^2
\end{pmatrix}\]

There is a linear map then corresponding to this matrix for any $(u_1, u_2) \in \NN^2$. 
\[ \KK(u_1, u_2)^{\oplus 3} \to \KK(u_1+1, u_2+1) \oplus \KK(u_1+2, u_2) \oplus \KK(u_1, u_2+2) \]

In general, suppose we choose some nonnegative integer $\bd(u) = \bd(u_1, ..., u_n)$ for each $u = (u_1, ..., u_n) \in \NN^n$ and some $\bd(v) = \bd(v_1, ..., v_n)$ for each $v = (v_1, ..., v_n) \in \NN^n$. Then there is a multigraded vector space map
\[ f: \KK(u_1, ..., u_n)^{\bd(u_1, ..., u_n)} \to \KK(v_1, ..., v_n)^{\bd(v_1, ..., v_n)} \]
corresponding to a path from $v$ to $u$ given by $x^{v-u}$, which has rank at most $\min\{\bd(u), \bd(v)\}$. Now, let $M$ be any finitely generated multigraded $R$-module. Let $u \in \NN^n$ and let $M_u = \KK(u)^{\bd(u)} = \KK(u_1, ..., u_n)^{\bd(u_1, ..., u_n)}$. Now, define the \textbf{top} of the module $M$ to be
\[
\ttop(M) = \bigoplus M_u = \bigoplus \KK(u)^{\bd(u)}:\ u \in \NN^n, \ \bd(u) \neq 0, \ \text{and} \ \not\exists \ v<u \in \NN^n: \ M_v \neq 0.
\]
Let us define a \textbf{projective cover} of $M$ by 
\[ P_0(M) = \bigoplus P_u^{\bd(u)} = \bigoplus R(-u)^{\bd(u)} \]
where $u$ and $\bd(u)$ are as in the definition of $\ttop(M)$. For an illustration of $\ttop(M)$ for a module over $\KK[x,y,z]$ see Figure \ref{3DStairs}. The corners of the boxes with bold vertices indicate $\ttop(M)$. 
\medskip

\begin{center}\label{3DStairs}
\fbox{Figure \ref{3DStairs}
\includegraphics[
    page=1,
    width=280pt,
    height=280pt,
    keepaspectratio
]{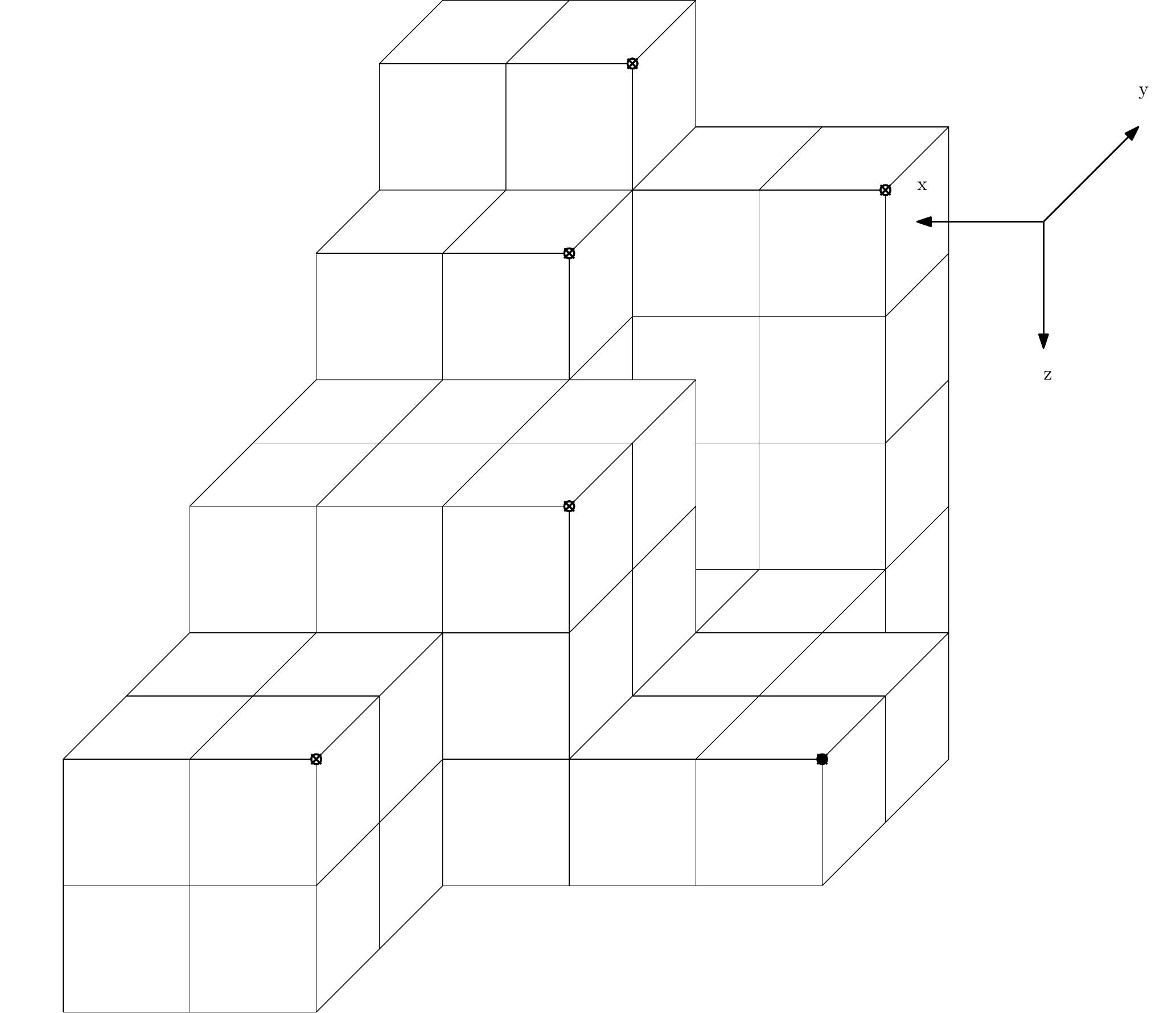}}
\end{center}
\medskip

Next, suppose in the partial order on $a, b \in \NN^n$ we have for any $R(-a)$ and $R(-b)$ that $c = \lcm(a,b)$. In other words, $c \in \NN^n$ corresponds to the monomial that is the least common multiple of the monomials corresponding to $a, b \in \NN^n$. 

\subsection{The Rank Invariant}

In this section a detailed description of the geometric properties of the rank invariant is given. Suppose $M$ is an $n$-graded persistence module, and the vector spaces $M_u = \KK(u)^{\bd(u)}$ and $M_v = \KK(v)^{\bd(v)}$ have dimension $\bd(u)$ and $\bd(v)$ respectively. Suppose further that $u \leq v$ in the partial order on $\NN^n$. The map $\phi_{u,v}(M) = x^{v-u} = x_1^{v_1-u_1} \cdots x_n^{v_n-u_n}: M_u \to M_v$ is a linear map of the form

\[ \phi(M) = \begin{pmatrix}
y_{1,1} & y_{1,2} & \cdots & y_{1,\bd(u)}\\
y_{2,1} & y_{2,2} & \cdots & y_{2,\bd(u)}\\
\vdots & \vdots & \ddots & \vdots \\
y_{\bd(v),1} & y_{\bf(v),2} & \cdots & y_{\bd(v),\bd(u)}
\end{pmatrix} \]

with $y_{i,j}$ indeterminates over $\KK$, i.e. it is a matrix in $\Mat_{\bd(v) \times \bd(u)}(\KK)$, with coordinates $(y_{i,j})$ once a basis of $M_u$ and $M_v$ have been chosen. Now, If we would like to classify all persistence modules with $\rank(\phi_{u,v}(M)) \leq r$, where $\phi(M)$ is the chosen linear map with respect to the module $M$, then we are geometrically classifying the so-called determinantal varieties. These varieties have a well known structure, but are by no means simple in general. However, many methods for understanding them are available, in particular, their coordinate rings have a standard monomial theory and Gr\"{o}bner bases given by filtrations using Schur functors and Young tableaux. 

\begin{defn}
A \textbf{determinantal variety} $\cV(m,n,r) \subset \Mat_{n \times m}(\KK)$ is the set of $n \times m$-matrices $X$ over $\KK$ (where $\KK = \ZZ$ or is a field), with rank $\rank(X) \leq r$. 
\end{defn}

We will study these varieties in relation to multiparameter persistence modules. In particular, we have the following obvious statement. 

\begin{prop}
Let $\cV(\bd(u), \bd(v), r)$ be the set of all morphisms of graded components of persistence modules $M_u \to M_v$, $\dim(M_u) = \bd(u)$ and $\dim(M_v) = \bd(v)$, and such that $\rank(\phi_{u,v}(M)) \leq r$, where $\phi_{u,v}(M): M_u \to M_v$. Then $\cV(\bd(u), \bd(v), r)$ is a determinantal variety. 
\end{prop}

\subsection{Determinantal Varieties}

Let $R = \KK[X_{i,j}]$ with $\KK$ any commutative ring (not necessarily $\ZZ$ or a field). In the case $\KK$ is a field, the space 
\[ X = \Hom_{\KK}(\KK^m, \KK^n) \]
is acted on by the algebraic group $\GL = \GL_n(\KK) \times \GL_m(\KK)$ by conjugation. The orbits are then simply given by $\rank(f) = r$ of any point $f \in X$. The orbit closures $X_r$ are given by all $g \in X$ such that $\rank(g) \leq \rank(f)$. However, if $\KK$ is not a field, the ideal of functions in $R$ vanishing on $X_r$ is $I_{r+1}$ generated by the $r+1$-order minors of $X_{i,j}$ and the coordinate ring is Cohen-Macaulay and normal. All of the ideals $I_r$ are $\GL$-invariant. The usual basis for the polynomial ring $R = \KK[X_{i,j}]$ is typically given by the $1 \times 1$-order minors $X_{i,j}$. In what follows, we will take a different basis for $R$, free over $\KK$, regardless of the characteristic of $\KK$.  

A classical characteristic free approach using the Schur functors allows one to define a basis of $\KK[E^* \otimes F]$ in terms of certain \emph{standard monomials} given by standard bi-tableaux. 

On the space $\Mat_{n \times m}(\KK)$ there is an action of $\GL_n(\KK) \times \GL_m(\KK)$. Orbits under this action are given by the rank of a matrix. The orbit closures (with respect to the Zariski topology) are exactly the \textbf{determinantal variety}. It is well known that such varieties are normal, Cohen-Macaulay, and have rational singularities. Let $\KK^{\bd(u)} = V(1),  \KK^{\bd(v)} = V(2)$. The defining ideal in the coordinate ring 
\[ \KK[\Hom(V(1), V(2)]/I = \Sym(V(1) \otimes V(2)^*)/I = \KK[\overline{\mathcal{O}_X}],\]
 is given by the $(r+1) \times (r+1)$ minors of the generic $\bd(v) \times \bd(u)$ matrix $X = (x_{ij})$ of $\bd(u) \cdot \bd(v)$ indeterminates. 

\begin{recall}
	For $X \subset \mathbb{A}_{\KK}^N$ an irreducible variety, the coordinate ring $k[X] = A/I$, where $A = \KK[x_1, ..., x_n]$ is the coordinate ring of the affine space, is \textbf{Cohen-Macaulay} if $\mathbf{pd}_A A/I = \mathit{codim} X$. For $S = \KK[X]$ a domain, $t \in S_{(0)}$ (localization of $S$ at $0$), and $t$ integral over $S$, if 
	\[ t^n+r_{n-1}t^{n-1} + \cdots r_0 = 0 \]
	with all $r_i \in S$, then $t \in S$. This is equivalent to $X$ being nonsingular in codimension $1$. 
\end{recall}

To summarize we have the following
\begin{theorem} (\cite{W1}):
	\begin{enumerate}
		\item The coordinate ring $\KK[\overline{\mathcal{O}_M}]$ is normal and Cohen-Macaulay. 
		\item $I_{r+1}$, the ideal given by the $(r+1) \times (r+1)$ minors of the generic matrix $X$, is the defining ideal of $\overline{\mathcal{O}_M}$. 
		\item The Hilbert function of $\KK[\overline{\mathcal{O}_X}]$ is independent of the characteristic of the base field $\KK$. 
	\end{enumerate}
\end{theorem}

\begin{theorem}
(\cite{DEP}): The coordinate ring $\KK[\overline{\mathcal{O}_X}]$ has a $\GL_{\bd(u)}(\KK) \times \GL_{\bd(v)}(\KK)$ filtration given by Schur functors of the form
\[ \bigoplus_{\lambda} S_{\lambda}M_u \otimes S_{\lambda} M_v^* \]
\end{theorem}

\begin{proof}
Let us first note the following isomorphisms,
\begin{align*}
\KK[\Hom_{\KK}(M_u, M_v)] &\cong \Sym(M_u \otimes M_v^*)\\
										 &\cong \bigoplus_{n \geq 0} \Sym^n(M_u \otimes M_v^*).
\end{align*}
Also, 
There is a natural filtration of $\Sym^n(M_u^* \otimes M_v)$ with the associated graded object
\[ \bigoplus_{|\lambda| = n} S_{\lambda}M_u\otimes S_{\lambda}M_v^*. \]
This is the well known \emph{Cauchy Formula} \cite{DEP} 
\end{proof}

We define a \textbf{standard bi-tableau}, denoted $(s|t)$ as a pair of standard tableaux $s$ and $t$ of the same shape $\lambda$, where by convention we write $s$ in reverse order. To a bitableau we associate a product of minors of $X$, where the entries of row $a$ of $s$ give the columns and the entries of row $a$ of $t$ give the rows of the minor of $X$. Then each pair of rows in $s$ and $t$ defines a minor of size $\lambda_a$, and the bi-tableau is associated to the product of these minors. It is well known that the standard bi-tableaux with at most $\min\{\dim V(1), \dim V(2)\} = \{\bd(u), \bd(v)\}$ columns forms a $\KK$-free basis of $\KK[V(1)^* \otimes V(2)]$, and that the standard bi-tableaux with at most $k$ columns form a $\KK$-free basis of $R/I_{k+1}$. We call the associated minors to a standard bi-tableau a \emph{standard monomial} in $R$ or $R/I_{k+1}$. It can be shown that any nonstandard monomial can be written as a sum of standard monomials each of which are earlier in the partial order on monomials. In particular if $m$ is nonstandard then we can write $m = \sum_i n_i m_i + \sum_jn_j m_j$, where $m_i$ are standard and of the same shape as the tableau corresponding to $m$, and the $m_j$ are standard but of a smaller shape in the order on tableau. \\
\\
Let $\Gr(n,m+n)$ denote the Grassmannian variety of $n$-dimensional subspaces of an $(m+n)$-dimensional vector space $V$. Let $\{w_1, ..., w_n\}$ be a basis for an $n$-dimensional subspace $W \subset V$. Then the vector $w_1 \wedge \cdots \wedge w_n$ determines a point $[W]$ in the projective space $\mathbb{P}(\bigwedge^nV)$, and the map $W \mapsto [W]$ is a one-to-one correspondence between $n$-dimensional subspaces of $V$ and points in $\mathbb{P}(\bigwedge^n V)$ (which are lines in $\bigwedge^nV$). For the standard basis $\{e_i\}_{i=1}^{m+n}$ of $V$, we have the standard basis $\{e_{i_1} \wedge \cdots \wedge e_{i_n}: i_1< \cdots < i_n\}$ of $\bigwedge^n V$. If to any $v_1 \wedge \cdots \wedge v_n \in \bigwedge^n V$ we associate the matrix 
\[X = \begin{pmatrix}
v_1 \\
\vdots \\
v_n
\end{pmatrix} = \begin{pmatrix}
v_{1,1} & v_{1,2} & \cdots & v_{1,m+n} \\
v_{2,1} & v_{2,2} & \cdots & v_{2,m+n} \\
\vdots & \vdots  & \ddots &  \vdots \\
v_{n,1} & v_{n,2} & \cdots & v_{n,m+n}
\end{pmatrix},\]
we have that in the basis $\{e_{i_1} \wedge \cdots \wedge e_{i_n}:i_1 < \cdots < i_n\}$ the coordinates of $v_1 \wedge \cdots \wedge v_n$ are given by the maximal minors of $X$. Let $[i_1, i_2, ..., i_n]$ denote the minor given by columns $i_1, i_2, ..., i_n$ of $X$. So,
\[ v_1 \wedge \cdots \wedge v_n = \sum_{1 \leq i_1 \leq \cdots \leq i_n \leq m+n} [i_1, ..., i_n]e_{i_1} \wedge \cdots \wedge e_{i_n}.\]
The open set $S_{n,m+1} = \{(v_1, ..., v_n) \in V^n: v_1 \wedge \cdots \wedge v_n \neq 0\} \subset \Hom(K^{m+n}, K^n)$ of maximal rank $n \times (m+n)$-matrices is called the \textbf{Stiefel manifold}. Now, $\Gr(n,m+n)$ can be identified with the orbits in $S_{n,m+n} \subset \Hom(\KK^{m+n}, \KK^n)$ under the action of $\GL_n(\KK)$ by left multiplication. Further, given a homomorphism $\pi: \KK^r \to \KK^{r+s}$ of two affine spaces, of the form
\[ \pi(x_1, ..., x_r) = (x_1, ..., x_r, p_1, ..., p_s) \]
where $p_i$ are polynomials in the $x_j$, the image of $\pi$ is a closed subvariety of $\KK^{r+s}$, and $\pi$ is an isomorphism of $\KK^r$ onto its image (it is the \emph{graph} of a polynomial map). The defining ideal is $I = (x_{r+i}-p_i)$, and the inverse map is
\[ (x_1, ..., x_r, ..., x_{r+s}) \mapsto (x_1, ..., x_r).\]
Now, consider the open set $U$ of $\Gr(n,m+n)$ where the Pl\"{u}cker coordinate given by the minor coming from the last $n$ columns is nonzero. $U$ can be identified with the space of $n \times m$ matrices. The association is given by associating any $X$ with the row span of the matrix
\[ (X| \tilde{\id_n}) = \begin{pmatrix}
x_{11} & \cdots & x_{1m} & 0 & \cdots & 0 & 1 \\
x_{21} & \cdots & x_{2m} & 0 & \cdots & 1 & 0 \\
\vdots & \ddots & \vdots & \vdots & \iddots   & \iddots & \vdots \\
x_{n1} & \cdots & x_{nm} & 1         &  0 & \cdots & 0 
\end{pmatrix}.\]
One may now identify Pl\"{u}cker coordinates with \textbf{bi-tableaux}. 
\begin{ex}\label{Plucker relations}
	Let $n=3$, $m=5$, and write
	\[ X = \begin{pmatrix}
	x_{31} & x_{32} & x_{33} & x_{34} & x_{35} & 1 & 0 & 0 \\
	x_{21} & x_{22} & x_{23} & x_{24} & x_{25} & 0 & 1 & 0 \\
	x_{11} & x_{12} & x_{13} & x_{14} & x_{15} & 0 & 0 & 1 \\
	\end{pmatrix}.\]
	Now let the Pl\"{u}cker coordinate $[1,2,3][1,2,5][1,2,7][2,3,6][5,6,7][5,6,8]$ be represented by the tableau
	\[ \young(123,125,127,238,567,568). \quad (*) \]
	To this we may associate the bi-tableau
	\[ \begin{pmatrix}
	\young(123,123,13,12,1,2),  \young(123,125,12,23,5,5)
	\end{pmatrix} \quad (**) \]
	Neither the tableau $(*)$, nor the double tableau $(**)$ are standard, and thus we may use the quadratic relations on Pl\"{u}cker coordinates to "straighten" the tableau $(*)$, which induces a straightening of the double tableau $(**)$ to a standard bi-tableau. These quadratic relations can be seen as the shuffling relations coming from the definition of the Schur functors from the previous section on Schur functors and Young tableaux. 
\end{ex}

We consider the Pl\"{u}cker coordinate standard if and only if the associated tableau is, i.e. the tableau is strickly increasing along rows, and nondecreasing along columns. In this example, the coordinate is not standard. We have a \textbf{straightening law} on the coordinates which is given by the definition of the Schur-functors and the relations among tableaux induced by the multilinear map defining them, i.e. the \textbf{shuffling relations}. One may also view the straightening law on standard bi-tableaux as a consequence of the quadratic relations on Pl\"{u}cker coordinates in the Grassmannian. If we take the perspective of double tableaux, we let the left tableau be the "row tableau", with indices $j \in [1,n]$ and the right tableau as the "column tableau" with indices $i \in [1,m+n]$. Each pair of rows gives a minor of the matrix $X$. We think of the space of one line tableau of size $k$ as a vector space $M_k$, with basis $(j_k, ..., j_i|i_1, ..., i_k)$, so that if two indices on the right or left are equal, then the symbol is zero, and the symbols are alternating separately on the left and right. For any partition $\lambda:= m_1 \geq m_2 \geq \cdots \geq m_r$, the tableaux of shape $\lambda$ can be thought of as a tensor product $M_{m_1} \otimes M_{m_2} \otimes \cdots \otimes M_{m_r}$. Evaluating a formal tableau as a product of minors gives a nontrivial kernel, which is the space spanned by the shuffling relations, or equivalently the straightening law. The action of $\GL(n) \times \GL(m)$ on $\KK[X]$ induces an action by the two groups of diagonal matrices, and the content of a bi-tableau is a weight vector for the product of the two algebraic tori. In particular If $(A,B) = (a_i) \times (b_j) \in T(n) \times T(m)$ is an element of the product of the tori $T(n) \subset \GL(n)$ and $T(m) \subset \GL(m)$, then the weight is $(\prod_{i+1}^n a_i^{-k(i)}; \prod_{j=1}^{m+n} b_j^{h(j)})$.

\subsection{Presentations of Persistence Modules}

Now, suppose we would like to study persistence modules with a given minimal free presentation
\[ F_1 \to F_0 \]
This can be broken down into the study of maps between two free multigraded modules. In particular, if we focus out attention on the maps $\ttop(F_1) \to \ttop(F_0)$, then we are in the situation of determinantal varieties,  since the maps of the tops can be represented as maps of the form
\[ \bigoplus_{u} M_u \to \bigoplus_v M_v \]
Now, these maps may not decompose into a direct sum of maps
\[ M_u \to M_v \]
so we will need to be careful in our study of them. In particular, suppose we have the following situation. 

\begin{center}
\begin{tikzcd}
{M(u_1, u_2+1)} &  &  &  &  &  &  \\
 &  &  &  &  &  &  \\
{M(u_1, u_2)} \arrow[uu] \arrow[rr] &  & {M(u_1+1, u_2)} \arrow[rr] &  & {M(u_1+2, u_2)} &  &  \\
 &  &  &  &  &  &  \\
 &  &  &  & {M(u_1+2, u_2-1)} \arrow[uu] \arrow[rr] &  & {M(u_1+3, u_2-1)}
\end{tikzcd}
\end{center}

Suppose that $\dim M(i,j) = \bd(i,j) = 1$ for each space in the diagram and that all maps are given by the identity. Then the map 
\[ \phi: M(u_1, u_2) \oplus M(u_1+2, u_2-1)  \oplus M(u_1+1, u_2) \to M(u_1, u_2+1) \oplus M(u_1+2, u_2) \oplus M(u_1+3, u_2-1) \oplus M(u_1+1, u_2) \]

can be represented as a vector space map
\[ \phi: \KK^6 \to \KK^6 \]
given by the map on the standard basis in the following diagram

\begin{center}
\begin{tikzcd}
e_1 & e_1 \\
e_2 \arrow[ru] \arrow[rd] & e_2 \\
e_3 \arrow[rd] & e_3 \\
e_4 & e_4 \\
e_5 \arrow[ru] \arrow[rd] & e_5 \\
e_6 & e_6
\end{tikzcd}
\end{center}
As endomorphisms on $\KK^6$ this is represented by the matrices
\[
X = \begin{pmatrix}
0 & 1 & 0 & 0 & 0 & 0 \\
0 & 0 & 0 & 0 & 0 & 0 \\
0 & 0 & 0 & 0 & 0 & 0 \\
0 & 0 & 0 & 0 & 1 & 0 \\
0 & 0 & 0 & 0 & 0 & 0 \\
0 & 0 & 0 & 0 & 0 & 0 
\end{pmatrix},  \quad \quad
Y = \begin{pmatrix}
0 & 0 & 0 & 0 & 0 & 0 \\
0 & 0 & 0 & 0 & 0 & 0 \\
0 & 1 & 0 & 0 & 0 & 0 \\
0 & 0 & 1 & 0 & 0 & 0 \\
0 & 0 & 0 & 0 & 0 & 0 \\
0 & 0 & 0 & 0 & 1 & 0 
\end{pmatrix}
\]
Note that this is an indecomposable finite dimensional representation of $\KK[x,y]/(xy)$. In particular, there is no way to decompose $V = \KK^6$ as a direct sum since $X$ and $Y$ cannot be simultaneously written as block diagonal matrices acting on complementary subspaces. Indeed, this would amount to dividing the matrix
\[
\begin{pmatrix}
0 & 1 & 0 & 0 & 0 & 0 \\
0 & 0 & 0 & 0 & 0 & 0 \\
0 & 1 & 0 & 0 & 0 & 0 \\
0 & 0 & 1 & 0 & 1 & 0 \\
0 & 0 & 0 & 0 & 0 & 0 \\
0 & 0 & 0 & 0 & 1 & 0 
\end{pmatrix}
\]
in such a way as to produce a block diagonal matrix.

\subsection{Fitting Ideals}

Proofs to the results which are not proven here can be found in \cite{E} Chapter $20$. Suppose $\phi: F \to G$ is a map of two free modules. Then $I_j(\phi)$ is the image of 
\[ \bigwedge^j F \otimes \bigwedge^j G^* \to S,\]
which is induced by 
\[ \bigwedge^j \phi : \bigwedge^j F \to \bigwedge^j G.\]
Choosing bases $\{f_i\}$ and $\{g_k\}$ for the free modules and representing $\phi$ in these bases as a matrix means that $I_j(\phi)$ can be realized as the ideal generated by the size $j$-minors of $\phi$. The convention we follow is that $I_0(\phi) = S$. The ideals $I_j(\phi)$ define invariants of finitely generated modules.

\begin{defn}
Let $M$ be a finitely generated module over a ring $S$ and let $\phi: F \to G \to M \to 0$ and $\phi': F' \to G' \to M \to 0$ be two free presentations such that $\rank(G) = r$ and $\rank(G') = r'$. Then for all $j \geq 0$ in $\ZZ$ we have 
\[ I_{r-j}(\phi) = I_{r'-j}(\phi ').\]
We define the $j^{th}$ \textbf{Fitting invariant} of $M$ to be the ideal 
\[ \Fitt_j(M) = I_{r-j}(\phi) \subset S.\]
\end{defn}

\begin{theorem}
The Fitting ideals commute with base change, so for any ring homomorphism $S \to S'$ we have 
\[ \Fitt_j(M \otimes_S S') = (\Fitt_j(M))S'.\]
\end{theorem}

\begin{theorem}
If the ring $(S,\fm)$ is local, then a module $M$ can be generated by $j$ elements if and only if $\Fitt_j(M)=S$. The closed subset of $\Spec(S)$ given by $\Fitt_j(M)$ is the set of all prime ideals $\fp$ such that $M_{\fp}$ cannot be generated by $j$ elements. 
\end{theorem}


%
\section{Generic Free Resolutions}\label{Generic Structure}

\subsection{Bruns' Generic Exactification}
Let $R = \KK[x_1, ..., x_n]$, and let 
\[
\xymatrix{
R^{b_0} & R^{b_{1}} \ar[l]_{x^{1}} & \cdots \ar[l]_{x^2} & R^{b_{n-1}} \ar[l]_{x^{n-1}} & R^{b_n}  \ar[l]_{x^n}& \ar[l] 0
}
\]
be a finite free resolution. The ranks of the cokernels of $x^i$ must be nonnegative, so 
\[
r_k = \sum_{j=k}^n (-1)^{k-j}b_j \geq 0
\]
for $k=0, ..., n$. We call this the \textbf{rank conditions} for $(b_0, b_1, ..., b_n)$.

\begin{theorem}
(Bruns \cite{B}): Let $(b_0, b_1 ..., b_n)$ be a sequence of nonnegative integers satisfying the rank condition $r=(r_1, ..., r_n)$, i.e. $b_k = r_k+r_{k-1}$, where we formally define $r_k = 0$ for $k \notin \{1, ..., n\}$. Then there exists a generic free resolution $(S, \FF_{\bullet})$ of type $(b_0, b_1 ..., b_n)$ in which $S$ is a countably generated $\ZZ$-algebra. 
\end{theorem}

The construction given by Bruns is as follows:

Start with the ring 
\[S_0 = \ZZ[X^k_{i(k),j(k)}]/\fa, \quad k=1, ..., n, \ i(k)=1, ..., b_k, \ j(k) = 1, ..., b_{k-1} \]
with $X_{i,j}^k$ a system of matrices of indeterminates over $\ZZ$, and $\fa$ is the ideal generated by the entries of the matrices given by the products $X^k X^{k-1}$, $k=2, ..., n$. Now, choose $\G_0$ the complex
\[
\G_0: \xymatrix{
S_0^{b_0} & S_0^{b_{1}} \ar[l]_{x^{1}} & \cdots \ar[l]_{x^2} & S_0^{b_{n-1}} \ar[l]_{x^{n-1}} & S_0^{b_n}  \ar[l]_{x^n}& \ar[l] 0
}
\]
Here $x^k$ is the matrix of residues of the entries in $X^k$. Next, we have the homology 
\[ \HH(\G_0): 
\xymatrix{
\HH(S_0^{b_n}) \ar[r] & \HH(S_0^{b_{n-1}}) \ar[r] & \cdots \ar[r] & \HH(S_0^{b_2}) \ar[r] & \HH(S_0^{b_1}) \ar[r] & \HH(S_0^{b_0})
}
 \]
where $\HH(S_0^{b_k}) = \ker(x^{k})/ \im(x^{k+1})$. Let $y_u^k = (y_{u,1}^k, ..., y_{u, b_k}^k)$ be the basis of $\HH(S_0^{b_k})$, where $u=1, ..., u_k$ and $k=1, ..., n$. Now, take 
\[S_1 = S_0[Z_u^{k,l}]/ \mathfrak{a}_1 \]
where $k=1, ..., n-1$, $u=1, ..., u_k$, $j=1, ..., b_k$. 
The polynomials
\[ y_{u,j}^k - \sum_{l=j}^{b_{k+1}}Z_u^{k,l}x_{l,j}^{k+1} \]
and the elements
\[ y_{u,j}^n, \ \text{with} \ u=1, ..., u_n, \ j=1, ..., b_n \]
generate $\fa_1$. Repeating this procedure for $S_1, S_2, ...$. 



\begin{remark}
It is important at this point to note what determines exactness of a complex. Let $R$ be a commutative ring, and $M^* = \Hom_R(M,R)$. We will say the \textbf{rank} of a projective module is $r$ if $\bigwedge^{r+1}P = 0$ and $\bigwedge^rP \neq 0$. $P$ has \textbf{well defined rank} if $\rank(P) = \rank(P_x)$ for every prime ideal $(x)$. If $f:P \to Q$ is a morphism of $R$-modules then $\bigwedge^k f : \bigwedge^k P \to \bigwedge^k Q$ induces a map
\[ \left(\bigwedge^k Q\right)^* \otimes \bigwedge^k P \to R\] 
with image $\fa(k)$. If $P$ and $Q$ are free modules of finite rank and we choose a basis then $\fa(k)$ is simply the $k \times k$ minors of the matrix representing $f$ in the chosen basis. Again, we say $\rank(f) = r$ if $\bigwedge^{r+1}f = 0$ but $\bigwedge^r f \neq 0$. This always exists and is finite if $P$ and $Q$ are finitely generated. Now, if $P$ is a free $R$-module and $\rank(P) = r$, we choose a generator of $\bigwedge^r P$, say $\alpha$ and define it to be the \textbf{orientation} of $P$. This determines an isomorphism
\[ \bigwedge^k P \cong \bigwedge^{r-k}P^*.\]
\begin{theorem}
(\cite{BE1}):
\end{theorem}
\end{remark}

Now, let
\[ S_1 = S_0[Z_u^{k,j}]/\fa_1 \]
where $k=1, ..., n-1$, $u=1, ..., u_k$, and $j=1, ..., b_{k+1}$, and where $Z_u^{k,j}$ are indeterminates over $S_0$. The ideal $\fa_1$ is generated by
\[ y_{u, p}^k - \sum_{j=p}^{b_{k+1}} Z_u^{k,j}x_{j,p}^{k+1}\] 
where $k=1, ..., n-1$, $u=1, ..., u_k$, $p=1, ..., b_k$. Continuing in this way to produce $S_2, S_3, ...., S_m$ we reach a point where $\G_m$ is an acyclic complex (i.e. exact except possibly at $S_m^{b_0}$). This gives the so-called \textbf{generic free resolution}. We will denote the generic free resolution with Betti numbers $(b_n, ..., b_0)$ as above by $R_{com}$ as it is the parametrizing ring of all such free resolutions. 

\section{Buchsbaum-Eisenbud Multipliers}\label{BE multipliers}

\subsection{Examples}\label{Exactify1}
\begin{ex}
The Macaulay2 code for the following example can be found in the Appendix \ref{Appendix} in Subsection \ref{Exactify1 Code}. 
Take $R = \ZZ[x_1, ..., x_6]$, and let 
\[M = \begin{pmatrix}{x}_{1}&
      {x}_{3}&
      {x}_{5}&
      {x}_{7}\\
      {x}_{2}&
      {x}_{4}&
      {x}_{6}&
      {x}_{8}\\
      \end{pmatrix}\]
be a generic matrix over $R$. Next let
\[ I = (-{x}_{2} {x}_{3}+{x}_{1} {x}_{4},-{x}_{2} {x}_{5}+{x}_{1} {x}_{6},-{x}_{4} {x}_{5}+{x}_{3} {x}_{6},-{x}_{2} {x}_{7}+{x}_{1}
      {x}_{8},-{x}_{4} {x}_{7}+{x}_{3} {x}_{8},-{x}_{6} {x}_{7}+{x}_{5} {x}_{8})\]
      be the ideal in $R$ generated by the $2 \times 2$-minors of $M$. Then $I$ has free resolution of the format,
     \[R \leftarrow R^{6} \leftarrow R^{8} \leftarrow R^{4} \leftarrow R \leftarrow 0.\]
     Next, we take the ring $\ZZ[Y_{i(k), j(k)}^k]_{k=1,...,4}$, where $Y^k$ is a matrix of $b_k \times b_{k-1}$ indeterminates over $\ZZ$ corresponding to the format $(b_0, b_1, b_2, b_3,b_4) = (1,6,8,4,1)$. Computing $Y^kY^{k-1}$ we get, 

\[ (Y^1Y^2)^t=\begin{pmatrix} 
	  {y}_{1} {y}_{7}+{y}_{2} {y}_{8}+{y}_{3} {y}_{9}+{y}_{4} {y}_{10}+{y}_{5} {y}_{11}+{y}_{6} {y}_{12} \\
      {y}_{1} {y}_{13}+{y}_{2} {y}_{14}+{y}_{3} {y}_{15}+{y}_{4} {y}_{16}+{y}_{5} {y}_{17}+{y}_{6} {y}_{18} \\
      {y}_{1} {y}_{19}+{y}_{2} {y}_{20}+{y}_{3} {y}_{21}+{y}_{4} {y}_{22}+{y}_{5} {y}_{23}+{y}_{6} {y}_{24} \\
      {y}_{1} {y}_{25}+{y}_{2} {y}_{26}+{y}_{3} {y}_{27}+{y}_{4} {y}_{28}+{y}_{5} {y}_{29}+{y}_{6} {y}_{30} \\
      {y}_{1} {y}_{31}+{y}_{2} {y}_{32}+{y}_{3} {y}_{33}+{y}_{4} {y}_{34}+{y}_{5} {y}_{35}+{y}_{6} {y}_{36} \\
      {y}_{1} {y}_{37}+{y}_{2} {y}_{38}+{y}_{3} {y}_{39}+{y}_{4} {y}_{40}+{y}_{5} {y}_{41}+{y}_{6} {y}_{42} \\
      {y}_{1} {y}_{43}+{y}_{2} {y}_{44}+{y}_{3} {y}_{45}+{y}_{4} {y}_{46}+{y}_{5} {y}_{47}+{y}_{6} {y}_{48} \\
      {y}_{1} {y}_{49}+{y}_{2} {y}_{50}+{y}_{3} {y}_{51}+{y}_{4} {y}_{52}+{y}_{5} {y}_{53}+{y}_{6} {y}_{54}
      \end{pmatrix}\]

$Y^2Y^3$ is too large to print on the screen and is shown in the Appendix \ref{Appendix} in Subsection \ref{Exactify1 Code}.

\[ Y^3Y^4=\begin{pmatrix}{y}_{55} {y}_{87}+{y}_{63} {y}_{88}+{y}_{71} {y}_{89}+{y}_{79} {y}_{90}\\
      {y}_{56} {y}_{87}+{y}_{64} {y}_{88}+{y}_{72} {y}_{89}+{y}_{80} {y}_{90}\\
      {y}_{57} {y}_{87}+{y}_{65} {y}_{88}+{y}_{73} {y}_{89}+{y}_{81} {y}_{90}\\
      {y}_{58} {y}_{87}+{y}_{66} {y}_{88}+{y}_{74} {y}_{89}+{y}_{82} {y}_{90}\\
      {y}_{59} {y}_{87}+{y}_{67} {y}_{88}+{y}_{75} {y}_{89}+{y}_{83} {y}_{90}\\
      {y}_{60} {y}_{87}+{y}_{68} {y}_{88}+{y}_{76} {y}_{89}+{y}_{84} {y}_{90}\\
      {y}_{61} {y}_{87}+{y}_{69} {y}_{88}+{y}_{77} {y}_{89}+{y}_{85} {y}_{90}\\
      {y}_{62} {y}_{87}+{y}_{70} {y}_{88}+{y}_{78} {y}_{89}+{y}_{86} {y}_{90}\\
      \end{pmatrix}\]

\end{ex}

\begin{ex}\label{BE multipliers 1}
Let $R = k[t_1, ..., t_n]$ be a polynomial ring, and let 

\[ A=\begin{pmatrix}{x}_{1}&
      {x}_{4}&
      {x}_{7}&
      {x}_{10}&
      {x}_{13}\\
      {x}_{2}&
      {x}_{5}&
      {x}_{8}&
      {x}_{11}&
      {x}_{14}\\
      {x}_{3}&
      {x}_{6}&
      {x}_{9}&
      {x}_{12}&
      {x}_{15}\\
      \end{pmatrix},
\quad  B= \begin{pmatrix} {y}_{1} &
      {y}_{6}\\
      {y}_{2}&
      {y}_{7}\\
      {y}_{3}&
      {y}_{8}\\
      {y}_{4}&
      {y}_{9}\\
      {y}_{5}&
      {y}_{10}\\
      \end{pmatrix}      \]


be a $3 \times 5$ generic matrix representing the map $d_1: R^5 \to R^3$ and $R^2 \to R^5$ in the standard basis $\{e_1, e_2, e_3\}$ and $\{f_1, ..., f_5\}$ and $\{g_1, g_2\}$.

\[ \left(\bigwedge^3(A)\right)^t = \begin{pmatrix} -{x}_{3} {x}_{5} {x}_{7}+{x}_{2} {x}_{6} {x}_{7}+{x}_{3} {x}_{4} {x}_{8}-{x}_{1} {x}_{6} {x}_{8}-{x}_{2} {x}_{4} {x}_{9}+{x}_{1} {x}_{5} {x}_{9} \\
      -{x}_{3} {x}_{5} {x}_{10}+{x}_{2} {x}_{6} {x}_{10}+{x}_{3} {x}_{4} {x}_{11}-{x}_{1} {x}_{6} {x}_{11}-{x}_{2} {x}_{4} {x}_{12}+{x}_{1} {x}_{5} {x}_{12} \\
      -{x}_{3} {x}_{8} {x}_{10}+{x}_{2} {x}_{9} {x}_{10}+{x}_{3} {x}_{7} {x}_{11}-{x}_{1} {x}_{9} {x}_{11}-{x}_{2} {x}_{7} {x}_{12}+{x}_{1} {x}_{8} {x}_{12} \\
      -{x}_{6} {x}_{8} {x}_{10}+{x}_{5} {x}_{9} {x}_{10}+{x}_{6} {x}_{7} {x}_{11}-{x}_{4} {x}_{9} {x}_{11}-{x}_{5} {x}_{7} {x}_{12}+{x}_{4} {x}_{8} {x}_{12} \\
      -{x}_{3} {x}_{5} {x}_{13}+{x}_{2} {x}_{6} {x}_{13}+{x}_{3} {x}_{4} {x}_{14}-{x}_{1} {x}_{6} {x}_{14}-{x}_{2} {x}_{4} {x}_{15}+{x}_{1} {x}_{5} {x}_{15} \\
      -{x}_{3} {x}_{8} {x}_{13}+{x}_{2} {x}_{9} {x}_{13}+{x}_{3} {x}_{7} {x}_{14}-{x}_{1} {x}_{9} {x}_{14}-{x}_{2} {x}_{7} {x}_{15}+{x}_{1} {x}_{8} {x}_{15} \\
      -{x}_{6} {x}_{8} {x}_{13}+{x}_{5} {x}_{9} {x}_{13}+{x}_{6} {x}_{7} {x}_{14}-{x}_{4} {x}_{9} {x}_{14}-{x}_{5} {x}_{7} {x}_{15}+{x}_{4} {x}_{8} {x}_{15} \\
      -{x}_{3} {x}_{11} {x}_{13}+{x}_{2} {x}_{12} {x}_{13}+{x}_{3} {x}_{10} {x}_{14}-{x}_{1} {x}_{12} {x}_{14}-{x}_{2} {x}_{10} {x}_{15}+{x}_{1} {x}_{11} {x}_{15} \\
      -{x}_{6} {x}_{11} {x}_{13}+{x}_{5} {x}_{12} {x}_{13}+{x}_{6} {x}_{10} {x}_{14}-{x}_{4} {x}_{12} {x}_{14}-{x}_{5} {x}_{10} {x}_{15}+{x}_{4} {x}_{11} {x}_{15} \\
      -{x}_{9} {x}_{11} {x}_{13}+{x}_{8} {x}_{12} {x}_{13}+{x}_{9} {x}_{10} {x}_{14}-{x}_{7} {x}_{12} {x}_{14}-{x}_{8} {x}_{10} {x}_{15}+{x}_{7} {x}_{11} {x}_{15}
      \end{pmatrix} =  
\begin{pmatrix} - {y}_{2} {y}_{6}+{y}_{1} {y}_{7}\\
      -{y}_{3} {y}_{6}+{y}_{1} {y}_{8}\\
      -{y}_{3} {y}_{7}+{y}_{2} {y}_{8}\\
      -{y}_{4} {y}_{6}+{y}_{1} {y}_{9}\\
      -{y}_{4} {y}_{7}+{y}_{2} {y}_{9}\\
      -{y}_{4} {y}_{8}+{y}_{3} {y}_{9}\\
      -{y}_{5} {y}_{6}+{y}_{1} {y}_{10}\\
      -{y}_{5} {y}_{7}+{y}_{2} {y}_{10}\\
      -{y}_{5} {y}_{8}+{y}_{3} {y}_{10}\\
      -{y}_{5} {y}_{9}+{y}_{4} {y}_{10}\\
      \end{pmatrix} = \bigwedge^{2}(B)
      \]
\end{ex}

%
\section{Varieties of Complexes of Persistence Modules}\label{varieties of complexes}

\subsection{Algebraic Geometry Prerequisites}
Before we move on to the more geometric approach to the study of multiparameter persistence modules, let us review a few notions from algebraic geometry and geometric invariant theory. Much of what follows is simply formal theory in order to ensure precision and justify a more intuitive approach in later sections. For the reader not concerned with such formal justifications, this section can be skipped on a first reading. Our main reference will be \cite{M} and \cite{H}, which we will follow closely. While the theory involved in the setup can feel quite daunting to a reader not familiar with algebraic geometry, it is not crucial to understanding and using later sections. In fact, once the formalities are taken care of, little more than a good understanding of the multilinear algebra which has already appeared in previous sections will be required. 

\begin{defn} (Algebraic Varieties): 
The most elementary and classical perspective on an algebraic variety is as the zero set of a collection of polynomial equations, say
\begin{align*}
p_1(x_1, ..., x_n) &=0 \\
p_2(x_1, ..., x_n) &=0 \\
\vdots \quad \quad \quad &  \quad \vdots \\
p_m(x_1, ..., x_n) &=0
\end{align*}
These equations will of course generate an ideal in $R = \KK[x_1, ..., x_n]$. However, we will need something more general for what follows since we will not be looking for solutions to the equations of polynomials $p \in \KK[x_1, ..., x_n]$ given by points in affine space $\mathbb{A}^n_{\KK}$. Instead, we will be looking for solutions in the more general "affine space" $\A^k_{R}$. To be precise, let $R$ be the $\KK$-algebra $\KK[x_1, ..., x_n]$. Define the set
\[ \cV_I(R) = \cV(R) = \{ y = (y_1, ..., y_k) \in R^k: \ p(y) = 0 \ \forall \ p \in I\},\]
where $p$ is a polynomial over $R$ in the ideal $I \subset R$, generated by polynomial equations:
\begin{align*}
p_1(y_1, ..., y_k) &=0 \\
p_2(y_1, ..., y_k) &=0 \\
\vdots \quad \quad \quad &  \quad \vdots \\
p_m(y_1, ..., y_k) &=0.
\end{align*}
Here, as usual, $R^k = R^{\oplus k} = R \oplus R \oplus \cdots \oplus R$ ($k$-summands). The set $\cV_R(I)$ is called the \textbf{$R$-valued points} of the \textbf{variety} $\cV_I = \cV$. We will generally omit the ideal $I$ if it is clear what $I$ is. 

Next, let $f:R \to S$ be a ring homomorphism to some commutative ring $S$ (generally this will just be another copy of $R$ in the following sections). If $y = (y_1, ..., y_k)$ is an $R$-valued point of $\cV$, then $f(y) = (f(y_1), ..., f(y_k))$ is an $S$-valued point (living in $\cV(S)$). So, $\cV(f): \cV(R) \to \cV(S)$. Also, given two ring homomorphisms, $f: R \to S$ and $g: S \to T$, we get $\cV(g \circ f) = \cV(g) \circ \cV(f)$. In particular, $\cV$ is a covariant functor from the categeory of $\KK$-algebras to the category of sets. The functor is entirely independent of coordinates. So, given an arbitrary $\KK$-algebra $\fR$, if $\cV(\KK) = \Spec(R/I)$ and $(a_1, ..., a_k) \in \fR$ is an $\fR$-valued point of $\cV$, then there is a homomorphism of algebras,
\[ f: R \to \fR \]
given by $x_i \mapsto a_i$ with $\ker(f) = I$, and so there is a homomorphism $R/I \to \fR$. So we identify
\[ \cV(\fR) = \Hom_{\KK}(R/I, \fR) = \Hom_{\KK}(\Spec(\fR), \cV(\KK)).\]
Again, if $f:\fR \to \mathfrak{S}$ is any ring homomorphism, then it induces a morphism of affine schemes
\[ f^{\#}:\Spec(\mathfrak{S}) \to \Spec(\fR)\]
in the usual way. Furthermore, it gives a set map
\[ \cV(f): \Hom_{\KK}(\Spec(\fR), \cV(\KK)) \to \Hom_{\KK}(\Spec(\mathfrak{S}), \cV(\KK)). \]
\[
\xymatrix{
\Spec(\mathfrak{S}) \ar[rr]^{f^{\#}} \ar[dr] & & \Spec(\fR) \ar[dl]_g \\
& \cV(\KK) &
}\]
\end{defn}

\begin{defn} (Algebraic Groups): 
Next, we will need the technology of \emph{algebraic groups}, which are just a generalization of linear groups (i.e. subgroups of invertible square matrices) to arbitrary commutative rings. In the classical case, one defines $\GL_n(\KK)$ to be the group of $n \times n$ matrices over the field $\KK$ which are invertible, i.e. such that $\det(g) \neq 0$ for all $g \in \GL_n(\KK)$. This group and the subgroup $\SL_n(\KK)$ of matrices with $\det(g) = 1$ will be important for what follows, but we will need more general versions. While for $\KK = \CC$, these groups carry the structure of a smooth manifold, if we replace $\KK$ with a commutative ring, for example $R = \KK[x_1, ..., x_n]$, they carry the structure of a scheme. To be precise, let 
\[ \SL_n(R) = \{ g = (g_{i,j}) \in R^{n^2}: \ \det(g) = \det(g_{i,j}) = 1\} .\]
This can be realized as scheme by identifying it with the subset $\{g=(g_{i,j}) \in R^{n^2}: \det(g)-1=0\}$. Intuitively, it is just the invertible matrices of determinant $1$, with entries in the ring $R$. We can treat $\SL_n$ as a functor in the same way we treated the varieties $\cV$ as functors, from the category of $\KK$-algebras to the category of sets. So we define an \textbf{algebraic group} to be an algebraic variety $G$, which is also a covariant functor from the category of $\KK$-algebras to the category of groups. 
\end{defn}

\begin{ex}
Let $A = \KK[x_{i,j}, \det(x)^{-1}]$ be the polynomial ring in $n^2$ variables with the inverse of the determinant $\det(x)^{-1} = \det(x_{i,j})^{-1}$ adjoined. We may think of this as localizing at $\det(x) \neq 0$. $\Spec(A) \subset \A^{n^2}_{\KK}$ is open in the Zariski topology, and gives the algebraic group $\GL_n(\KK)$. 
\end{ex}

\subsection{Varieties of Complexes of Persistence Modules}
Let $\bk$ be the ring $\ZZ$ or a field and let $R = \bk[x_1, ..., x_n]$ be the polynomial ring with coefficients in $\bk$. Let 
\[
\xymatrix{
R^{b_0} & \ar[l]_{d_1} R^{b_1} & \ar[l]_{d_2} \cdots  & \ar[l]_{d_{n-1}} R^{b_{n-1}} & \ar[l]_{d_n} R^{b_n}
}
\]
be a sequence of finite free $R$-modules $F_k = R^{b_k}$, and where the Betti numbers $(b_0, b_1, ..., b_n)$ denote $\rank(F_k) = b_k$. Denote the $R$ module given by all $R$-homomorphisms $f: F_k \to F_{k-1}$, for $k=1, ..., n$ by, 
\[ \Hom_R (F_k, F_{k-1}) \cong F_{k}^* \otimes F_{k-1}. \]
Define the $R$-module
\[ A^N = \bigoplus_{k=1}^{n} \Hom_R (F_k, F_{k-1}) \cong \bigoplus_{k=1}^n F_k^* \otimes F_{k-1},\]
where $N = \sum_{k=1}^n b_{k-1}b_k$. This can be thought of as an affine space similar to the case when $R$ is replaced by a field $\bk$, and the free 
modules $F_k = R^{b_k}$ are replaced by $\bk$-vector spaces $V_k = \bk^{b_k}$. Then $\Hom_{\bk}(V_k, V_{k-1})$ is 
isomorphic to the affine space $\Mat_{b_{k-1} \times b_k}(\bk)$, of $b_{k-1} \times b_k$ matrices over $\bk$, and $A^N 
\cong \bigoplus_{k=1}^n \bk^{b_{k-1} \times b_k}$ as a $\bk$-vector space. If we take an $n$-tuple of maps $(f_1, f_2, ..., f_n) \in A^N$, this defines a point in $A^N$. Let us suppose we take all such $n$-tuples of maps in $A^N$ such that $f_k \circ f_{k-1} = 0$. This defines an affine scheme as follows. Since $\Spec(\bk[x_1, ..., x_n]) = \bk^n$, every ring homomorphism $f: R \to R$ corresponds to a morphism of the affine space $\phi: \bk^n \to \bk^n$, and every morphism $f_k: R^{b_k} \to R^{b_{k-1}}$ corresponds to a morphism $\phi_k: \prod_{j(k-1)=1}^{b_{k-1}} \bk^n \to \prod_{j(k)=1}^{b_k} \bk^n$. 

We may view the $R$-linear maps $f_k: R^{b_k} \to R^{b_{k-1}}$ as matrices of indeterminates over $R$ once a basis $\{e(k)_1, ..., e(k)_{b_k}\}$ of each free module $F_k = R^{b_k}$ is chosen and fixed. We then identify the spaces $\Hom_R(F_i, F_{i-1})$ with matrices 
\[ X(i)_{j(i), k(i)} = \begin{pmatrix}
x(k)_{1,1} & x(k)_{1,2} & \cdots & x(k)_{1,b_{k-1}} \\
x(k)_{2,1} & x(k)_{2,2} & \cdots & x(k)_{2,b_{k-1}} \\
\vdots & \vdots & \ddots & \vdots \\
x(k)_{b_k,1} & x(k)_{1,2} & \cdots & x(k)_{b_k,b_{k-1}} 
\end{pmatrix} \cong F_k^* \otimes F_{k-1}
\]
where in the chosen basis we have the bijection sending $e(k)_{i(k)}^* \otimes e(k-1)_{j(k)}$ to the matrix with a $1$ in the $(i(k),j(k))$ entry, and zeros elsewhere, $k=1, ..., n$.

Then, taking the collection of all $n$-tuples of matrices $(X_1, X_2, ..., X_n)$ such that $X_kX_{k-1}=0$, we get a collections of polynomial equations over $R$ whose zero set defines an affine scheme. We will call this a \textbf{variety of complexes (of persistence modules)}. For the Betti numbers $b = (b_0, ..., b_n)$ we denote this variety by $\cV(b)$. Define $\rank(f_k)$ of a map (respectively $\rank(X_k)$ of a matrix) to be the largest positive integer $r_k$ such that $\bigwedge^{r_k} f_k \neq 0$ (resp. $\bigwedge^{r_k} X_k \neq 0$), but $\bigwedge^{r_k+1}f_k = 0$ (resp. $\bigwedge^{r_k+1}X_k = 0$). It is not difficult to see, that if $(X_1, ..., X_n)$, corresponding to $(f_1, ..., f_n)$, is a point in $\cV(b)$, then we must have $\rank(f_k)+\rank(f_{k-1}) \leq b_k$. This yields another set of polynomial equations given by
\[ \bigwedge^{r_k+1} X_k = 0\]
for each choice of $n$-tuple of ranks $r = (r_1, ..., r_n)$. We define this closed subset by $\cV(b,r)$. Clearly, depending on the choice of $b = (b_0, b_1, ..., b_n)$ we may have several choices of \textbf{maximal rank sequences}, i.e. $n$-tuples of ranks $r=(r_1, ..., r_n)$ such that $r_k+r_{k-1} \leq b_k$, but such that no $r_k$ can be increased such that the sequence still satisfies these conditions.

\subsection{Rank Conditions and Orbits}

Given $X = (X_1, ..., X_n) \in A^N$, the set $\cV(b,r) = \{X \in A^N: \rank(X_i) \leq r_i\}$ corresponding to the rank conditions $r = (r_1, ..., r_n)$, is closed in the Zariski topology. Moreover, it is clear that if $r' = (r_1', ..., r_n')$ is such that $r_i' \leq r_i$ for $i=1, ..., n$, then the set $\cV(b,r')$ is a closed subset of $\cV(b,r)$. We may define a partial order on rank sequences by defining $r'=(r_1', ..., r_n') \leq r=(r_1, ..., r_n)$ if and only if $r_i' \leq r_i$ for all $i=1, ..., n$. Further, if $r$ and $r'$ are incomparable in this partial order, then $\cV(b,r)$ and $\cV(b.r')$ have intersection $\cV(b,r'')$ where $r''$ is the rank sequence satisfying $r_i'' \leq r_i, r_i'$. We may define a partial order on the $\cV(b,r)$ in this way given by containment. 

We will very suggestively call $\cV(b,r)$ an \textbf{orbit closure} in $A^N$, and we will call the set
\[ \cO(b,r) = \{X \in A^N: \rank(X(i)) = r_i\} \]
an \textbf{orbit} in $A^N$. In the Zariski topology, $\cO(b,r)$ is open in $\cV(b,r)$, and the closure is $\overline{\cO(b,r)} = \cV(b,r)$. We will say that $\cV(b,r)$ \textbf{degenerates} to $\cV(b,r')$ if $\cV(b,r') \subset \cV(b,r)$, and we say an orbit $\cO(b,r)$ \textbf{degenerates} to the orbit $\cO(b,r')$ if $\cO(b,r')$ is in the closure $\cO(b,r) = \cV(b,r)$. This yields a partial order on orbits which we call the \textbf{degeneration order}.

\subsection{Coordinate Rings}

Next, we would like to study the coordinate rings $\scO(\cV(b))$ and $\scO(\cV(b,r))$. In particular, we want to show that

\begin{prop}
Given Betti numbers $b=(b_0, b_1, ..., b_n)$ and rank conditions $(r_1, ..., r_n)$ we have isomorphisms, 
\begin{enumerate}
\item $\scO(\cV(b)) = k[X(i)_{j(i), k(i)}]/I(b)$, where $I = (X(i)X(i-1))_{i=1,...,n}$, and
\item $\scO(\cV(b,r)) = \scO(\cV(b))/I(r)$, where $I(r)$ is generated by the polyonomial equations given by $\bigwedge^{r_i+1}X_i$. 
\end{enumerate}
\end{prop}

To prove this, we provide an equivariant filtration of the coordinate rings via Schur functors. In particular, we show there is a basis of the rings given by \emph{standard multitableaux}. We then show this basis provides a standard monomial theory and yields a Gr\"{o}bner basis of the coordinate rings.

%
\section{Bases of Coordinate Rings via Young Tableaux}\label{Bases and Standard Monomials}

\subsection{Standard Monomials and Gr\"{o}bner Bases}
We now define a Gr\"{o}bner basis for the coordinate rings of the Buchsbaum-Eisenbud varieties of complexes. Much of modern commutative algebra and algebraic geometry is formulated in a nonconstructive manner. Finding practical algorithms allows one to apply computer algebra packages such as Macaulay2 to various problems. One extremely useful and central tool is that of Gr\"{o}bner bases. Gr\"{o}bner bases allow many questions about ideals in polynomial rings to be reduced to monomial ideals, which are much easier to handle. 

\begin{defn}
	Let us write monomials in the ring $R = \KK[x_1, ..., x_n]$ as a vector $a = (a_1, ..., a_n)$, which will represent the monomial $x^{a_1} \cdots x^{a_n}$. Any ideal generated by such monomials will be called a \textbf{monomial ideal}. If $J \subset R$ is a monomial ideal, then the set of all monomials not in $J$ forms a vector space basis of $R/J$. For arbitrary ideals $I \subset R$ one would like to find a similar description of the basis of $R/I$. Any $R/I$ will have a monomial basis, and if we choose this basis $B$ to be the complement of some monomial ideal $J$, we are able to easily determine when a monomial is in $B$, since $J$ is generated by finitely many monomials by simply determining if it is divisible by a generator of $J$. We define a \textbf{monomial order} on an ideal $I \subset R$ to be a total order $>$ on the monomials of $I$ such that if $m_1, m_2$ are two monomials of $I$ and $n \neq 1$ is a monomial of $R$, then
	\[ m_1 > m_2 \implies nm_1 > nm_2.\]
	Any such monomial order is \textbf{artinian}, i.e. every subset has a least element. 
\end{defn}

\begin{defn}
	If $>$ is a monomial order on $I \subset R$, then for any $f \in I$ we define the \textbf{initial term of} $f$, denoted by $\init(f)$, to be the greatest monomial term of $f$. We define the \textbf{reverse lexicographic order} by $a = (a_1, ..., a_m) > b = (b_1, ..., b_n)$ if and only if the degree of $a$ is larger than the degree of $b$, or the degrees are equal and $a_i<b_i$ for the last index $i$ with $a_i \neq b_i$. In other words, we compare total degree first, then in reverse order compare $a_n, b_n$, then $a_{n-1}, b_{n-1}$, and so on. When we reach an $a_i \neq b_i$ in the reverse order, then $a_i < b_i$ implies $a>b$. 
\end{defn}

Let $d_1, d_2$ be two positive integers and let 
\[ V(d_1, d_2, r) = \{(A,B): A\subset [1,d_1], \ B \subset [1,d_2] ; 1 \leq |A|=|B| \leq r\}.\]
For $k=1, ..., N$, let $V_k = \{ \la A|B \ra : (A,B)_k \in V(d_{k-1}, d_k, r_k)\}$ define a set of indeterminates. 
This notation $\la A,B \ra$ is meant to represent the $r \times r$-minor of a matrix given by the rows $A$, and columns $B$ ($|A| = |B| = r$). Now, define a total order on $V_k$ for all $k \in [1,r]$ by 
\[ \la A|B \ra < \la C|D \ra \iff A<C,\ \text{or} \ A=C \ \text{and}\ B<D.\]
Now, define $V(\bd) = V_1 \amalg \cdots V_N$ to be the disjoint union, and let $R = k[x:x \in V(\bd)]$. Assign degree $1$ to each variable $x$. Let
\[ N^{V(\bd)} = \N^{V_1} \times \cdots \times \N^{V_N}. \]
Later, we will think of $\bd = (d_0, ..., d_N)$ as a dimension vector of a type $\A_{N+1}$ quiver (with relations), and each $r_k$ will be a rank condition on the matrix $X_k$ associated to the arrow $a_k$ in a representation of this quiver. Extend the orders on the $V_i$, to a monomial order on $\N^{V(\bd)}$ using the \emph{reverse lexicographic order} on each $\N^{V_i}$ and the \emph{lexicographic product order} on $\N^{V(\bd)}$. Now, we define a partial order on finite subsets of $\N$ by
\[ \{a_1 < \cdots < a_s\} \leq \{b_1 < \cdots < b_t\} \iff s \leq t, \ \text{and} \ a_i \geq b_i, i=1, ....,s, \]
then define a partial order on the $V_k$ by
\[ \la A|B \ra \leq \la C|D \ra \iff A \leq C \ \text{and} \ B\leq D.\]

This is then extended to the lexicographic product order on each $V_k$. This can be visualized as a sequence of lattices, each representing the partial order on each $V_k$, a vertex being a monomial, and a path connecting points of each lattice. If two paths are of equal length and differ at some point, then the first place they differ determines the "lower path" in the order. Obviously since monomials can be of arbitrary length the complete lattice for each $V_k$ cannot be drawn, but for every monomial (which is necessarily finite) the lattice of elements above it in the partial order may be drawn. 

\begin{ex}
If \[X = \begin{pmatrix}
x_{11} & x_{12} & x_{13} \\
x_{21} & x_{22} & x_{23}
\end{pmatrix},\]
then the linear order on the minors is
\[ \la 2|3 \ra < \la 2|2 \ra < \la 2|1 \ra < \la 1|3 \ra < \la 1|2 \ra < \la 1|1 \ra  \]
\[ < \la 1,2|2,3 \ra < \la 1,2|1,3 \ra < \la 1,2|1,2 \ra
\]
which can be represented as a column of nine vertices with arrows connecting elements immediately lower in the order. The partial order on the minors is given by
\medskip
\[
\begin{tikzcd}
 &  & {\langle 1,2|1,2 \rangle} \arrow[d] &  \\
 &  & {\langle 1,2|1,3 \rangle} \arrow[ld] \arrow[rd] &  \\
 &  \langle 1|1 \rangle \quad \arrow[ld] \arrow[rd] &  & {\langle 1,2|2,3 \rangle} \arrow[ld] \\
\langle 2|1 \rangle \quad \arrow[rd] &  & \langle 1|2 \rangle \arrow[ld] \arrow[rd] &  \\
 & \langle 2|2 \rangle \quad \arrow[rd] &  & \quad \langle 1|3 \rangle \arrow[ld] \\
 &  & \langle 2|3 \rangle & 
\end{tikzcd}
\]
\medskip
\end{ex}

\begin{ex}\label{2wedge of A}
As another example, let 
\[ A = \begin{pmatrix}{x}_{1}&
      {x}_{4}&
      {x}_{7}&
      {x}_{10}&
      {x}_{13}\\
      {x}_{2}&
      {x}_{5}&
      {x}_{8}&
      {x}_{11}&
      {x}_{14}\\
      {x}_{3}&
      {x}_{6}&
      {x}_{9}&
      {x}_{12}&
      {x}_{15}\\
      \end{pmatrix} \]
Then we have 
\[\left(\bigwedge^2(A)\right)^t = \begin{pmatrix}-{x}_{2} {x}_{4}+{x}_{1} {x}_{5}&
      -{x}_{3} {x}_{4}+{x}_{1} {x}_{6}&
      -{x}_{3} {x}_{5}+{x}_{2} {x}_{6}\\
      -{x}_{2} {x}_{7}+{x}_{1} {x}_{8}&
      -{x}_{3} {x}_{7}+{x}_{1} {x}_{9}&
      -{x}_{3} {x}_{8}+{x}_{2} {x}_{9}\\
      -{x}_{5} {x}_{7}+{x}_{4} {x}_{8}&
      -{x}_{6} {x}_{7}+{x}_{4} {x}_{9}&
      -{x}_{6} {x}_{8}+{x}_{5} {x}_{9}\\
      -{x}_{2} {x}_{10}+{x}_{1} {x}_{11}&
      -{x}_{3} {x}_{10}+{x}_{1} {x}_{12}&
      -{x}_{3} {x}_{11}+{x}_{2} {x}_{12}\\
      -{x}_{5} {x}_{10}+{x}_{4} {x}_{11}&
      -{x}_{6} {x}_{10}+{x}_{4} {x}_{12}&
      -{x}_{6} {x}_{11}+{x}_{5} {x}_{12}\\
      -{x}_{8} {x}_{10}+{x}_{7} {x}_{11}&
      -{x}_{9} {x}_{10}+{x}_{7} {x}_{12}&
      -{x}_{9} {x}_{11}+{x}_{8} {x}_{12}\\
      -{x}_{2} {x}_{13}+{x}_{1} {x}_{14}&
      -{x}_{3} {x}_{13}+{x}_{1} {x}_{15}&
      -{x}_{3} {x}_{14}+{x}_{2} {x}_{15}\\
      -{x}_{5} {x}_{13}+{x}_{4} {x}_{14}&
      -{x}_{6} {x}_{13}+{x}_{4} {x}_{15}&
      -{x}_{6} {x}_{14}+{x}_{5} {x}_{15}\\
      -{x}_{8} {x}_{13}+{x}_{7} {x}_{14}&
      -{x}_{9} {x}_{13}+{x}_{7} {x}_{15}&
      -{x}_{9} {x}_{14}+{x}_{8} {x}_{15}\\
      -{x}_{11} {x}_{13}+{x}_{10} {x}_{14}&
      -{x}_{12} {x}_{13}+{x}_{10} {x}_{15}&
      -{x}_{12} {x}_{14}+{x}_{11} {x}_{15}\\
      \end{pmatrix}\]

\end{ex}

We follow Pragacz and Weyman in \cite{PW} to define a set of \textbf{standard monomials} in $\N^{V(\bd)}$. Let $m = m_1 \cdots m_N \in \N^{V(\bd)}$ ($m_i \in \N^{V_i}$). Write $m_i = m_{i1}> \cdots > m_{is_i}$, so that $m_{ij} = \la A_{ij}|B_{ij} \ra_i$, and $m_{i1} \geq \cdots \geq m_{is_i}$. Now let $t(A_i) = (A_{i1}, ..., A_{is_i})$, and
\[ t(B_i) =
\begin{cases}
\emptyset      & \quad \text{if } i=n \ \text{and} \ B_{Nj}=\{1,...,d_N\} \ \text{for} \ j=1,...,s_N\\
((B_{Ns_N})^c, ..., (B_{Np_N})^c)  & \quad \text{if} \  i=n \ \text{and}\ p_N = \min\{j:[1,d_N] \neq B_{Nj} \\
((B_{is_i})^c, ..., (B_{i1})^c) & \quad \text{if} \ i \neq N
\end{cases}
\]
and define the tableau for $m$ as
\[ t(m) = \begin{pmatrix}
\emptyset & t(B_1) & t(B_2) & \cdots & t(B_{n-1}) & t(B_n) \\
t(A_1) & t(A_2) & t(A_3) & \cdots & t(A_n) & \emptyset 
\end{pmatrix}. \]

\begin{defn}
	A monomial $m \in \N^{V(\bd)}$ will be a \textbf{standard monomial} if $t(w)$ is a standard multitableau, and if the elements of $V^{\max}(\bd) = \{\la A|B \ra_k: 1 \leq k \leq N, |A|=|B|=r_k \}$ do not divide $m$. 
\end{defn}

\begin{remark}
	Let $\lambda(B_i)$ denote the Young diagram associated to $t(B_i)$. Computing the entries of $t(B_i)$ amounts to computing the image of the equivariant isomorphism
	\[ \bigwedge^{\lambda(B_i)}F_i^* \to \bigwedge^{\lambda(B_i)^*}F_i \]
	where $\lambda(B_i)^*$ is the "\textbf{dual partition}", or the complementary partition in the $|\lambda(B_i)| \times \bd(i)$ rectangle. For notation see \S \ref{Young Tableaux}, and \cite{W}. 
\end{remark}

\begin{ex}
	Let $\bd = (2,5,3)$, $r = (r_1, r_2) = (2,3)$, and let $m = \la 1,2|1,4 \ra_1 \la 2|3 \ra_1 \la 1,3|2,3 \ra_2$. Then
	\[ t(m) = \left( \ 
	\young(12,2),\  \young(1245,235,13), \ \young(13) \right). \]
	Since this is not a standard multi-tableau, $m$ is not a standard monomial. 
\end{ex}

\begin{ex}
	Suppose again that $\bd = (2,5,3)$, $r = (r_1, r_2) = (2,3)$, and let $m = \la 1,2|1,4 \ra_1 \la 2|3 \ra_1 \la 2,3|2,3 \ra_2$. Thus, 
	\[ t(m) = \left( \ 
	\young(12,2),\  \young(1245,235,23), \ \young(23) \right). \]
	This is a standard multitableau, but $\la 1,2|1,4 \ra_1|m$ and $\la 1,2|1,4 \ra_1 \in V^{\max}(\bd)$, thus $m$ is still not a standard monomial. If we take $m = \la 2|3 \ra_1 \la 2,3|2,3 \ra_2$, then $m$ is standard. 
\end{ex}

\begin{defn}
Now, as in \ref{BE multipliers}, we will need the Buchsbaum-Eisenbud multipliers. Recall, these can be described by minors as follows. Let $A \subset \{1, ..., b_{k-1}\}$ such that $|A| = r_k$, and let 
\[ M_k = \{ \la A \ra_k : \ A \subset \{1, ..., b_{k-1}\}, \ |A|=r_k\} \]
be a set of indeterminates. We have from Section \ref{BE multipliers} that there are unique $\la A \ra_k \in R$ such that
\[ \la A|E \ra -\sgn(E^c,E) \la A \ra_k \la E^c \ra_{k+1} = 0 \in R \]
for all $E \subset \{1, ..., b_k\}$, $|E|=r_k$. These are precisely the \emph{Buchsbaum-Eisenbud multipliers} for the complex $\cF_{\bullet}$ of format $(b_0, ..., b_N)$ with ranks $r = (r_1, ..., r_N)$. 
\end{defn}
\begin{theorem}
(\cite{T1} Proposition $1.3$): The following expressions are equal to $0 \in R$,
\begin{enumerate}
\item \[ \sum_{(C \cap A) \subset \Gamma \subset (C-D)}^{|\Gamma|=q} \sgn(A,C-\Gamma) \sgn(C-\Gamma, \Gamma) \sgn(\Gamma, D) \cdot \la A \cup (C-\Gamma) \ra_k \la \Gamma \cup D \ra_k \]
\item \[ \sum_{(C \cap A) \subset \Gamma \subset (C-E)}^{|\Gamma|=q} \sgn(A,C-\Gamma) \sgn(C-\Gamma, \Gamma) \sgn(\Gamma, E) \cdot \la A \cup (C-\Gamma) \ra_k \la \Gamma \cup E| F  \ra_k \]
\item \[ \sum_{\Gamma \subset \{1, ..., b_k\}-(\Lambda \cup H \cup K)}^{|\Gamma|=t} \sgn(H,\Gamma) \sgn(\Gamma, K) \la G| H \cup \Gamma \ra_k \la \Gamma \cup K \ra_{k+1} .\]
\end{enumerate}
Here we have $A, C, D, E, G \subset \{1, ..., b_{k-1}\}$, $F, H, K, \Lambda \subset \{1,...,b_k\}$, and
\[ |A|=r_k-p \quad |F|=s \leq r_k \quad |G| = m \leq r_k \]
\[ |D| = r_k-q \quad |E|=s-q \quad |H|=m-t \]
\[ |C| = p+q \geq r_k+1 \quad |K|=r_{k+1}-t \quad |\Lambda| < t \leq \min\{m, r_{k+1}\} \]
\end{theorem}

\begin{theorem}
(\cite{T1} Lemma $1.8$): Let $x^i$ be a matrix in the complex of free modules once a basis is fixed,
\[
\xymatrix{
R^{b_0} & R^{b_{1}} \ar[l]_{x^{1}} & \cdots \ar[l]_{x^2} & R^{b_{n-1}} \ar[l]_{x^{n-1}} & R^{b_n}  \ar[l]_{x^n}& \ar[l] 0
}
\]
and denote by $\im(x^i)$ the image of $x^i$ in $R^{b_{i-1}}$. Let $\Frac(R)$ be the total ring of fractions of $R$. Assuming each ideal of minors $I_{r_i}(x^i)$ contains an $R$-regular element, we have that
\[ \im(x^i) \otimes_R \Frac(R) \]
is free of rank $r_i$ with free $\Frac(R)$-basis so that for all $A \subset \{1, ..., b_{i-1}\}$, $|A|=r_i$, the Buchsbaum-Eisenbud multiplier $\la A \ra_i \in R$ is the maximal minor $\la A| 1, ..., r_i \ra_{y^i}$, of the $b_{i-1} \times r_i$ matrix $y^i$ of the map
\[ \im(x^i) \otimes_R \Frac(R) \to R^{b_{i-1}} \otimes_R \Frac(R).\]
Moreover, the relations
\[ 
\sum_{(C \cap A) \subset \Gamma \subset (C-D)}^{|\Gamma|=q} \sgn(A,C-\Gamma) \sgn(C-\Gamma, \Gamma) \sgn(\Gamma, D) \cdot \la A \cup (C-\Gamma) \ra_k \la \Gamma \cup D \ra_k
\]
are the Pl\"{u}cker relations on $y^i$ as discussed in \ref{Plucker relations}. 
\end{theorem}

\begin{defn}\label{Standard Monomials}
	Now, define a homomorphism of $k$-algebras $\pi: R \to S$, where $R = k[x:x \in V(\bd)]$ and $S$ is the ring corresponding to 
	the quotient of $\Sym(V_0 \otimes V_{1}^* \oplus \cdots \oplus V_{n-1} \otimes F_n^*)$, by the relations induced by the representations
	\[ V_{i} \otimes V_{i+2}^* \hookrightarrow (V_{i} \otimes V_{i+1}^*) \otimes (V_{i+1} \otimes V_{i+2}^*) \]
	corresponding to the conditions $d_{i}d_{i+1} = 0$, and the representations
	\[ \bigwedge^{r_i+1}V_i \otimes \bigwedge^{r_i+1}V_{i+1}^* \]
	corresponding to the condition $\bigwedge^{r_i+1}d_i = 0$. The homomorphism is defined by sending $\la A|B \ra_k$ to the corresponding minor of $X_k \in \Hom(V_{k-1}, V_k)$. The map $\pi$ is surjective. The map $\pi$ maps the standard monomials bijectively to a free basis of $S$ as a $k$-vector space. This follows directly from the facts described in \S \ref{Straightening Law 1} and results of Pragacz and Weyman in \cite{PW}. If we let $\Sigma(\bd)$ be the set of \textbf{nonstandard monomials} then $\Sigma(\bd)$ is a monomial ideal in $\N^{V(\bd)}$.
\end{defn}

\begin{theorem}
	(Tchernev): $\Sigma(\bd)$ is the initial ideal of $I \subset k[X_1,...,X_n]$, where $I$ is the ideal generated by
	\[V_{i} \otimes V_{i+2}^* \hookrightarrow (V_{i} \otimes V_{i+1}^*) \otimes (V_{i+1} \otimes V_{i+2}^*) \]
	and 
	\[ \bigwedge^{r_i+1}V_i \otimes \bigwedge^{r_i+1}V_{i+1}^*.\]
\end{theorem}

We may study the singularities of $S$ via the combinatorics of simplicial ideals. Let $\Delta(\bd)$ be the simplicial complex with vertex set $V(\bd)$, and with faces $F \subset V(\bd)$ given by $m_F = \prod_{v \in F}v$, such that the product is standard. 

\begin{defn}
	A simplicial complex $\Delta$ of dimension $d$ is \textbf{constructible} if:
	\begin{enumerate}
		\item $\Delta$ is a simplex; or
		\item there exist proper $d$-dimensional constructible subcomplexes $\Delta_1, \Delta_2 \subset \Delta$ such that $\Delta_1 \cap \Delta_2$ is constructible of dimension $d-1$ and $\Delta_1 \cup \Delta_2 = \Delta$. 
	\end{enumerate}
\end{defn}

\begin{theorem}
	(Tchernev): The simplicial complex $\Delta(\bd)$ is constructible. 
\end{theorem}

\section{Examples}
\begin{ex}
	Suppose that we have a product of three varieties of complexes, represented by the following diagram:
	\[ \xymatrix{ \bullet_1 \ar@/^/@{.>}[r]  \ar[r] \ar@/_/@{~>}[r] & \bullet_2 \ar@/^/@{.>}[r]  \ar[r] \ar@/_/@{~>}[r] & \bullet_3 \ar@/^/@{.>}[r]  \ar[r] \ar@/_/@{~>}[r] & \bullet_4 } \]
	Let us denote the "colors" by the set $\{1, 2, 3\}$ (ordered from top to bottom). Let $\{a_1, a_2, a_3\}$ be the arrows of color "$1$", labeled from left to right. Similarly, label arrows of color "$2$" by $\{b_j\}_{j=1, 2, 3}$, and arrows of color "$3$" by $\{c_k\}_{k=1, 2, 3}$. Take the dimension vector $\bd = (2,4,5,3)$, and rank maps $r_1 = (2,2,3) = r_2, r_3 = (1,3,2) $. 
	
	The associated graded object for the coordinate ring is
	\begin{align*}
	\gr\left(\KK[A_i, B_j, C_k]_{i,j,k =1, 2, 3} / I\right) &= S_{(\lambda(a_1))}V_1 \otimes S_{(\lambda(a_2),-\lambda(a_1))}V_2 \otimes 
	S_{(\lambda(a_3), -\lambda(a_2))}V_3 \otimes S_{(-\lambda(a_3))}V_4 \\
	& \otimes  S_{(\mu(b_1))}V_1 \otimes S_{(\mu(b_2),-\mu(b_1))}V_2 \otimes S_{(\mu(b_3), 
		-\mu(b_2))}V_3 \otimes S_{(-\mu(b_3))}V_4 \\
	& \otimes  S_{(\nu(c_1))}V_1 \otimes S_{(\nu(c_2),-\nu(c_1))}V_2 \otimes S_{(\nu(c_3), 
		-\nu(c_2))}V_3 \otimes S_{(-\nu(c_3))}V_4 .
	\end{align*}
	Using the dimension vector $(2, 4, 5, 3)$ and the rank sequences $r_1 = (2,2,3) = r_2$, and $r_3 = (1,3,2)$, then the maps $\lambda, \mu, \nu:Q_1 \to \mathcal{P}$, which assign a partition (or Young diagram) to each arrow, are restricted to partitions such that $\lambda(a_i)$ and $\mu(b_j)$ have no more than $r_1(i) = r_2(j)$ (for $i=j$) parts. Similarly, $\nu(c_k)$ is restricted to partitions which have no more than $r_3(k)$ parts. 
	
	Let $m_1 = \la 3|2 \ra_1^1 \la 2|2 \ra_1^1 \la 3|4 \ra_2^1 \la 2,3|2,4 \ra_3^1 \in k[\rep_{Q, 1}(\bd_1, r_1)]$, $m_2 = \la 2|2 \ra_2^2 \la 1,2|1,2 \ra_3^2 \in k[\rep_{Q, 2}(\bd_2, r_2)]$, and $m_3 = \la 2, 3 | 2, 3 \ra_2^3 \la 1,2|2,3 \ra_3^3  \in k[\rep_{Q, 3}(\bd_3, r_3)]$. Then the monomial $m_1 \otimes m_2 \otimes m_3 \in k[\rep_{Q, c}(\bd, r)]$ corresponds to the product of multitableaux, 
	\[ \left(\young(2,3),\  \young(134,134,3), \ \young(1235,23),\ \young(13) \right) \times \left(\emptyset ,\  \young(2), \ \young(1345,12), \ \young(3) \right) \times \left(\emptyset,\  \young(23), \ \young(145,12), \ \young(1) \right). \]
	This is \emph{not standard} since $m_2$ and $m_3$ are not standard, thus this is an element of $\Sigma(\bd) = \Sigma(\bd_1) + \Sigma(\bd_2) + \Sigma(\bd_3)$. In the partial order we have defined, it is larger than the monomial
	\[ \la 3|2 \ra_1^1 \la 2|2 \ra_1^1 \la 3|4 \ra_2^1 \la 2,3|2,4 \ra_3^1 \otimes  \la 2|3 \ra_2^2 \la 1,2|1,2 \ra_3^2 \otimes \la 2, 3 | 2, 3 \ra_2^3 \la 1,2|2,3 \ra_3^3,\]
	which differs in only one place from the $m_2$ factor. 
\end{ex}

\section{Applications: Ising Model Hamiltonian}

From Xanadu's \emph{Pennylane} software tutorials for quantum computing \cite{Xanadu}, consider the class of qubit Hamiltonians that are quadratic, meaning that the terms of the Hamiltonian represent either interactions between two qubits, or the energy of individual qubits. This class of Hamiltonians is naturally described by graphs, with second-order terms between qubits corresponding to weighted edges between nodes, and first-order terms corresponding to node weights. A well known example of a quadratic Hamiltonian is the transverse-field Ising model, which is defined as

\[
H_{\text{Ising}}(\mathbf{\theta}) = \sum_{(i,j) \in E} \theta_{i,j}^{(1)}Z_i \otimes Z_j + \sum_{i \in V} \theta_i^{(2)}Z_i + \sum_{i \in V}X_i  
\] 

where $\mathbf{\theta} = \{ \theta^{(1)}, \theta^{(2)} \}$. In this Hamiltonian, the set $E$ that determines which pairs of qubits have $Z_i \otimes Z_j$ interactions can be represented by the set of edges for some graph. With the qubits as nodes, this graph is called the interaction graph. The $\theta^{(1)}$ parameters correspond to the edge weights and the $\theta^{(2)}$ parameters correspond to weights on the nodes. If we allow a time varying interaction graph (the $\mathbf{\theta}$ terms) we have a way of applying topological data analysis (persistent homology) to the interaction graph. As the time parameter increases, we apply an edge dropout to edges below a given interaction threshold, to be thought of as a second parameter, time being the first. If three nodes are connected to each other with sufficient strength, we fill in the face of the triangle. If four or more are connected, we fill in the tetrahedron, etc. This provides us with a two parameter persistent homology model with time and bonding strength as our two parameters, and with the interaction graph associated to the time varying Hamiltonian. The software package \href{https://github.com/rivetTDA}{\textbf{Rivet}} built for $2$-parameter persistent homology (see \cite{Lesnick-Wright-et al.}) can then be used to determined various data such as the presentations of the persistence modules discussed in \S \ref{presentation spaces}

As an application of the theory developed thus far we propose a way of defining a Hamiltonian Ansatz for arbitrary materials by defining a multiscale, $2$-parameter Hamiltonian Ansatz for an arbitrary material via persistence homology given by a dynamic DB-Scan (density based scan) algorithm related to a proposal presented in \cite{WXZ} and \cite{AJLMWX}. This method will work for a useful scale-dependent definition of a Hamiltonian and a definition of entanglement that can be adjusted for emergent topological, structural, and dynamic patterns, properties, and features of the material at different scales. 

In the quantum computing version of this problem, using arbitrary controlled-U gates in place of controlled-X (CNOT) gates to entangle qubits we get a more complex variety of entangled states. The tunable controlled-U gate parameters can be interpreted as a form of distance between qudits. Using energy level of qudits, with a notion of distance between two qudits defined in terms of energy level as a second parameter and using a dynamic DB-scan with each of these two parameters gives a way of computing 2-parameter persistent homology. The persistent homology allows us to understand emergent geometric and topological properties of the quantum system.

%
\section{Appendix: Macaulay2 Code}\label{Appendix}

\textcolor{darkred}{\hrule}
\subsection{Example \ref{BE multipliers 1}}

\begin{lstlisting}
i1 : R=ZZ[x_1..x_15, y_1..y_10];

i2 : A=genericMatrix(R,x_1,3,5);

             3       5
o2 : Matrix R  <--- R

i3 : B=genericMatrix(R,y_1,5,2);

             5       2
o3 : Matrix R  <--- R

i4 : A




o4 = | x_1 x_4 x_7 x_10 x_13 |
     | x_2 x_5 x_8 x_11 x_14 |
     | x_3 x_6 x_9 x_12 x_15 |

             3       5
o4 : Matrix R  <--- R

i5 : B

o5 = | y_1 y_6  |
     | y_2 y_7  |
     | y_3 y_8  |
     | y_4 y_9  |
     | y_5 y_10 |

             5       2
o5 : Matrix R  <--- R

i6 : A*B

o6 = | x_1y_1+x_4y_2+x_7y_3+x_10y_4+x_13y_5 x_1y_6+
															  x_4y_7+x_7y_8+x_10y_9+x_13y_10 |
     | x_2y_1+x_5y_2+x_8y_3+x_11y_4+x_14y_5 x_2y_6+
     														  x_5y_7+x_8y_8+x_11y_9+x_14y_10 |
     | x_3y_1+x_6y_2+x_9y_3+x_12y_4+x_15y_5 x_3y_6+
     														  x_6y_7+x_9y_8+x_12y_9+x_15y_10 |

             3       2
o6 : Matrix R  <--- R

i7 : exteriorPower(3,A)

o7 = | -x_3x_5x_7+x_2x_6x_7+x_3x_4x_8-x_1x_6x_8-x_2x_4x_9+x_1x_5x_9 -
	x_3x_5x_10+x_2x_6x_10+x_3x_4x_11-x_1x_6x_11-x_2x_4x_12+x_1x_5x_12
     ----------------------------------------------------------------------
     --------------------------------------------------------------------
     -x_3x_8x_10+x_2x_9x_10+x_3x_7x_11-x_1x_9x_11-x_2x_7x_12+x_1x_8x_12 -
     x_6x_8x_10+x_5x_9x_10+x_6x_7x_11-x_4x_9x_11-x_5x_7x_12+x_4x_8x_12
     ----------------------------------------------------------------------
     --------------------------------------------------------------------
     -x_3x_5x_13+x_2x_6x_13+x_3x_4x_14-x_1x_6x_14-x_2x_4x_15+x_1x_5x_15 -
     x_3x_8x_13+x_2x_9x_13+x_3x_7x_14-x_1x_9x_14-x_2x_7x_15+x_1x_8x_15
     ----------------------------------------------------------------------
     --------------------------------------------------------------------
     -x_6x_8x_13+x_5x_9x_13+x_6x_7x_14-x_4x_9x_14-x_5x_7x_15+x_4x_8x_15
     --------------------------------------------------------------------------
     ----------------------------------------------------------------
     -x_3x_11x_13+x_2x_12x_13+x_3x_10x_14-x_1x_12x_14-x_2x_10x_15+x_1x_11x_15
     --------------------------------------------------------------------------
     ----------------------------------------------------------------
     -x_6x_11x_13+x_5x_12x_13+x_6x_10x_14-x_4x_12x_14-x_5x_10x_15+x_4x_11x_15
     --------------------------------------------------------------------------
     ----------------------------------------------------------------
     -x_9x_11x_13+x_8x_12x_13+x_9x_10x_14-x_7x_12x_14-x_8x_10x_15+x_7x_11x_15 |

             1       10
o7 : Matrix R  <--- R

i8 : exteriorPower(2,B)

o8 = | -y_2y_6+y_1y_7  |
     | -y_3y_6+y_1y_8  |
     | -y_3y_7+y_2y_8  |
     | -y_4y_6+y_1y_9  |
     | -y_4y_7+y_2y_9  |
     | -y_4y_8+y_3y_9  |
     | -y_5y_6+y_1y_10 |
     | -y_5y_7+y_2y_10 |
     | -y_5y_8+y_3y_10 |
     | -y_5y_9+y_4y_10 |

             10       1
o8 : Matrix R   <--- R
-------------------------------------------------------------
\end{lstlisting}

\medskip
\textcolor{darkred}{\hrule}
\medskip

\subsection{Exactification Example \ref{Exactify1}}\label{Exactify1 Code}
\begin{flushleft}
\begin{lstlisting}
i1 : R=ZZ[x_1..x_8]

o1 = R

o1 : PolynomialRing

i2 : M=genericMatrix(R,x_1,2,4)

o2 = | x_1 x_3 x_5 x_7 |
     | x_2 x_4 x_6 x_8 |

             2       4
o2 : Matrix R  <--- R

i3 : I=minors(M)
stdio:3:3:(3): error: no method found for applying minors to:
     argument   :  | x_1 x_3 x_5 x_7 | (of class Matrix)
                   | x_2 x_4 x_6 x_8 |

i4 : I=minors(2,M)

o4 = ideal (- x x  + x x , - x x  + x x , - x x  + x x , - x x  + x x , - x x  + x x , - x x  + x x )
               2 3    1 4     2 5    1 6     4 5    3 6     2 7    1 8     4 7    3 8     6 7    5 8

o4 : Ideal of R

i5 : res(I)

      1      6      8      4      1
o5 = R  <-- R  <-- R  <-- R  <-- R  <-- 0
                                         
     0      1      2      3      4      5

o5 : ChainComplex

i6 : F=res(I)

      1      6      8      4      1
o6 = R  <-- R  <-- R  <-- R  <-- R  <-- 0
                                         
     0      1      2      3      4      5

o6 : ChainComplex

i7 : F.dd








o7 = 
              1                                             6
	 0 : R  <----------------------------------------- R  : 1
               | -x_2x_3+x_1x_4 -x_2x_5+x_1x_6 -x_4x_5+x_3x_6 
               -x_2x_7+x_1x_8-x_4x_7+x_3x_8 -x_6x_7+x_5x_8 |

          6                                             8
     1 : R  <----------------------------------------- R  : 2
               {2} | x_6  x_5  x_8  x_7  0    0    0    0    |
               {2} | -x_4 -x_3 0    0    0    0    x_8  x_7  |
               {2} | x_2  x_1  0    0    x_8  x_7  0    0    |
               {2} | 0    0    -x_4 -x_3 0    0    -x_6 -x_5 |
               {2} | 0    0    x_2  x_1  -x_6 -x_5 0    0    |
               {2} | 0    0    0    0    x_4  x_3  x_2  x_1  |

          8                                             4
     2 : R  <----------------------------------------- R  : 3
               {3} | -x_8 -x_7 0    -x_2x_7+x_1x_8 |
               {3} | 0    -x_8 -x_7 -x_2x_8        |
               {3} | x_6  x_5  0    x_2x_5-x_1x_6  |
               {3} | 0    x_6  x_5  x_2x_6         |
               {3} | x_2  x_1  0    0              |
               {3} | 0    x_2  x_1  x_2^2          |
               {3} | -x_4 -x_3 0    -x_2x_3+x_1x_4 |
               {3} | 0    -x_4 -x_3 -x_2x_4        |

          4                    1
     3 : R  <---------------- R  : 4
               {4} | x_1  |
               {4} | -x_2 |
               {4} | 0    |
               {5} | 1    |

          1
     4 : R  <----- 0 : 5
               0

o7 : ChainComplexMap

i8 : d1=F.dd_1

o8 = | -x_2x_3+x_1x_4 -x_2x_5+x_1x_6 -x_4x_5+x_3x_6 -x_2x_7
			+x_1x_8 -x_4x_7+x_3x_8 -x_6x_7+x_5x_8 |

             1       6
o8 : Matrix R  <--- R

i9 : d2=F.dd_2

o9 = {2} | x_6  x_5  x_8  x_7  0    0    0    0    |
     {2} | -x_4 -x_3 0    0    0    0    x_8  x_7  |
     {2} | x_2  x_1  0    0    x_8  x_7  0    0    |
     {2} | 0    0    -x_4 -x_3 0    0    -x_6 -x_5 |
     {2} | 0    0    x_2  x_1  -x_6 -x_5 0    0    |
     {2} | 0    0    0    0    x_4  x_3  x_2  x_1  |

             6       8
o9 : Matrix R  <--- R

i10 : d3=F.dd_3

o10 = {3} | -x_8 -x_7 0    -x_2x_7+x_1x_8 |
      {3} | 0    -x_8 -x_7 -x_2x_8        |
      {3} | x_6  x_5  0    x_2x_5-x_1x_6  |
      {3} | 0    x_6  x_5  x_2x_6         |
      {3} | x_2  x_1  0    0              |
      {3} | 0    x_2  x_1  x_2^2          |
      {3} | -x_4 -x_3 0    -x_2x_3+x_1x_4 |
      {3} | 0    -x_4 -x_3 -x_2x_4        |

              8       4
o10 : Matrix R  <--- R

i11 : d4=F.dd_4

o11 = {4} | x_1  |
      {4} | -x_2 |
      {4} | 0    |
      {5} | 1    |

              4       1
o11 : Matrix R  <--- R

i17 : S0=ZZ[y_1..y_90];

i20 : D1=genericMatrix(S0,y_1,1,6)

o20 = | y_1 y_2 y_3 y_4 y_5 y_6 |

               1        6
o20 : Matrix S0  <--- S0

i21 : D2=genericMatrix(S0,y_7,6,8)

o21 = | y_7  y_13 y_19 y_25 y_31 y_37 y_43 y_49 |
      | y_8  y_14 y_20 y_26 y_32 y_38 y_44 y_50 |
      | y_9  y_15 y_21 y_27 y_33 y_39 y_45 y_51 |
      | y_10 y_16 y_22 y_28 y_34 y_40 y_46 y_52 |
      | y_11 y_17 y_23 y_29 y_35 y_41 y_47 y_53 |
      | y_12 y_18 y_24 y_30 y_36 y_42 y_48 y_54 |

               6        8
o21 : Matrix S0  <--- S0

i22 : D3=genericMatrix(S0,y_55,8,4)

o22 = | y_55 y_63 y_71 y_79 |
      | y_56 y_64 y_72 y_80 |
      | y_57 y_65 y_73 y_81 |
      | y_58 y_66 y_74 y_82 |
      | y_59 y_67 y_75 y_83 |
      | y_60 y_68 y_76 y_84 |
      | y_61 y_69 y_77 y_85 |
      | y_62 y_70 y_78 y_86 |

               8        4
o22 : Matrix S0  <--- S0

i23 : D4=genericMatrix(S0,y_87,4,1)

o23 = | y_87 |
      | y_88 |
      | y_89 |
      | y_90 |

               4        1
o23 : Matrix S0  <--- S0

i24 : I12=D1*D2

o24 = | y_1y_7+y_2y_8+y_3y_9+y_4y_10+y_5y_11+y_6y_12
	  y_1y_13+y_2y_14+y_3y_15+y_4y_16+y_5y_17+y_6y_18 y_1y_19
	 +y_2y_20+y_3y_21+y_4y_22+y_5y_23+y_6y_24
      -----------------------------------------------------------
      y_1y_25+y_2y_26+y_3y_27+y_4y_28+y_5y_29+y_6y_30 
      y_1y_31+y_2y_32+y_3y_33+y_4y_34+y_5y_35+y_6y_36 y_1y_37
      +y_2y_38+y_3y_39+y_4y_40+y_5y_41+y_6y_42
      -----------------------------------------------------------
      y_1y_43+y_2y_44+y_3y_45+y_4y_46+y_5y_47+y_6y_48 
      y_1y_49+y_2y_50+y_3y_51+y_4y_52+y_5y_53+y_6y_54 |

               1        8
o24 : Matrix S0  <--- S0

i25 : I23=D2*D3



               6        4
o25 : Matrix S0  <--- S0

i26 : I34=D3*D4

o26 = | y_55y_87+y_63y_88+y_71y_89+y_79y_90 |
      | y_56y_87+y_64y_88+y_72y_89+y_80y_90 |
      | y_57y_87+y_65y_88+y_73y_89+y_81y_90 |
      | y_58y_87+y_66y_88+y_74y_89+y_82y_90 |
      | y_59y_87+y_67y_88+y_75y_89+y_83y_90 |
      | y_60y_87+y_68y_88+y_76y_89+y_84y_90 |
      | y_61y_87+y_69y_88+y_77y_89+y_85y_90 |
      | y_62y_87+y_70y_88+y_78y_89+y_86y_90 |

               8        1
o26 : Matrix S0  <--- S0


i30 : exteriorPower(3,D3)



               56        4
o30 : Matrix S0   <--- S0
\end{lstlisting}
\end{flushleft}

\medskip
\textcolor{darkred}{\hrule}
\medskip

\subsection{Example \ref{2wedge of A}}

\begin{lstlisting}
----------------------------------------------------
i1 : R=ZZ[x_1..x_15, y_1..y_10];

i2 : A=genericMatrix(R,x_1,3,5);

             3       5
o2 : Matrix R  <--- R
i3 : transpose(exteriorPower(2,A))




o3 =  {-2} | -x_2x_4+x_1x_5     -x_3x_4+x_1x_6     -x_3x_5+x_2x_6     |
      {-2} | -x_2x_7+x_1x_8     -x_3x_7+x_1x_9     -x_3x_8+x_2x_9     |
      {-2} | -x_5x_7+x_4x_8     -x_6x_7+x_4x_9     -x_6x_8+x_5x_9     |
      {-2} | -x_2x_10+x_1x_11   -x_3x_10+x_1x_12   -x_3x_11+x_2x_12   |
      {-2} | -x_5x_10+x_4x_11   -x_6x_10+x_4x_12   -x_6x_11+x_5x_12   |
      {-2} | -x_8x_10+x_7x_11   -x_9x_10+x_7x_12   -x_9x_11+x_8x_12   |
      {-2} | -x_2x_13+x_1x_14   -x_3x_13+x_1x_15   -x_3x_14+x_2x_15   |
      {-2} | -x_5x_13+x_4x_14   -x_6x_13+x_4x_15   -x_6x_14+x_5x_15   |
      {-2} | -x_8x_13+x_7x_14   -x_9x_13+x_7x_15   -x_9x_14+x_8x_15   |
      {-2} | -x_11x_13+x_10x_14 -x_12x_13+x_10x_15 -x_12x_14+x_11x_15 |

              10       3
o3 : Matrix R   <--- R
----------------------------------------------------
\end{lstlisting}

\medskip
\textcolor{darkred}{\hrule}
\medskip


\end{document}